\documentclass[10pt] {article}  

\usepackage{amsmath}
\usepackage{amssymb}
\usepackage{amsthm}
\usepackage{amsfonts} 				

\usepackage[sfdefault=cmbr]{isomath}

\usepackage[all,cmtip]{xy}

\def\mb{\mathbb}

\def\Z{\mathbb Z}
\def\Q{{\mathbb Q}}

\theoremstyle{theorem}
\newtheorem{theorem}{Theorem}[subsection]
\newtheorem{proposition}{Proposition}[subsection]

\newtheorem{corollary}{Corollary}[subsection]
\theoremstyle{definition}
\newtheorem{definition}{Definition}[subsection]
\newtheorem{remark}{Remark}[subsection]

\title{A modular approach to Thue-Mahler equations}
\author{Dohyeong Kim}

\begin{document}

\maketitle
\begin{abstract}
Let $h(x,y)$ be a non-degenerate binary cubic form with integral coefficients, and let $S$ be an arbitrary finite set of prime numbers. By a classical theorem of Mahler, there are only finitely many pairs of relatively prime integers $x,y$ such that $h(x,y)$ is an $S$-unit. In the present paper, we reverse a well known argument, which seems to go back to Shafarevich, and use the modularity of elliptic curves over $\Q$ to give upper bounds for the number of solutions of such a Thue-Mahler equation. In addition, our methods gives an effective method for determining all solutions, and we use Cremona's Elliptic Curve Database to give a wide range of numerical examples.
\end{abstract}

\tableofcontents
\section{Introduction}\label{section:1}
Let $h(x,y)$ be a non-degenerate cubic form with integer coefficients, and let $S$ be a finite set consisting of $s$ distinct prime numbers, say $p_1,p_2,\cdots, p_s $. Then the Thue-Mahler equation
\begin{align}
h(x,y) = \pm \prod_{i=1}^{s}p_i^{e_i}
\end{align}
has finitely many solutions among relatively prime integers $x,y$ and non-negative integers $e_1,e_2, \cdots, e_s$. In geometric terms, if we denote by $\Z_S$ the ring of $S$-integers, and by $Y$ the affine variety defined as the complement of zeros of $h(x,y)$ in a projective line, then the solutions of above Thue-Mahler equations, modulo the identification of $(x,y)$ and $(-x,-y)$, bijectively correspond to the elements of $Y(\Z_S)$.
\par
Mahler gave an ineffective proof of the finiteness of $Y(\Z_S)$, and Coates \cite{Coates I}, \cite{Coates II}, \cite{Coates III} later obtained an effective finiteness of $Y(\Z_S)$  using Baker's estimate of linear forms in logarithms together with its $p$-adic analogues. However, explicit determination of $Y(\Z_S)$ based on Baker's method is often practically impossible due to the astronomical size of resulting upper bound for the height of a putative solution $t \in Y(\Z_S)$. 
\par
The aim of the current article is to present a new approach to Thue-Mahler equation. In order compute $Y(\Z_S)$, we design a descent procedure, which mimics the Kummer homomorphism for rational points on elliptic curves. More precisely, we will construct a natural map
\begin{align}
\kappa	\colon Y(\Z_S) &\longrightarrow \{\text{Elliptic Curves over $\Q$ up to isomorphism}\} 
\\
 t &\longmapsto X_t
\end{align}
which associates an elliptic curve $X_t$ to an unknown solution $t \in Y(\Z_S)$, and study local properties of $X_t$. In particular, we will show that $X_t$ has good reduction outside of $S$, discriminant of $h(x,y)$, and $2$. It allows one to compute the image of $\kappa$ without knowing elements of $Y(\Z_S)$. On the other hand, for an elliptic curve $E$, we will show that $\kappa^{-1}(E)$ is naturally a zero dimensional algebraic variety defined by explicit polynomials with rational coefficients, whose $\Q$-points correspond $t \in Y(\Z_S)$ equipped with an isomorphism from $X_t$ to $E$. In particular, one can numerically compute $\kappa^{-1}(E)$ from coefficients of a Weierstrass equation for $E$.
\par
Existence of such a map $\kappa$ first allows us to bound the cardinality of $Y(\Z_S)$ from above, in terms of the number of elliptic curves whose conductor belongs to a finite list of integers, where the list of possible conductors is obtained from the coefficients of $h(x,y)$ and $S$. The number of such elliptic curves can be bounded from above either using the work \cite{Brumer Silverman} of Brumer and Silverman or modularity of elliptic curves. The former has better asymptotics, while the latter provides a practical algorithm.
\begin{theorem}
Let $S$ be a finite set of primes containing $2$ and prime divisors of the discriminant of $h(x,y)$. Let $G(S)$ be the number of isomorphism classes of elliptic curves which have good reduction outside of $S$. Then we have
\begin{align}
Y(\Z_S) \le \left| \mathrm{Aut}_\Q(Y) \right| \times G(S).
\end{align}
Combining it with an upper bound for $G(S)$, due to Brumer and Silverman, we have
\begin{align}
Y(\Z_S) \le \left| \mathrm{Aut}_\Q(Y) \right| \times k_2 M^{\frac 1 2 + \varepsilon}.
\end{align}
where $M$ is the product of all prime numbers in $S$, $\varepsilon$ is an arbitrary positive number, and $k_2$ is a constant depending on $\varepsilon$. 
\end{theorem}
In fact, one can identify $\mathrm{Aut}_\Q(Y) $ with a subgroup of the symmetric group acting on zeros of $h(x,y)$, so it has at most six elements.
\par
We would like to stress that our proof is manifestly constructive, which ultimately relies on the modularity of elliptic curves defined over rational numbers. More precisely, we give an explicit characterisation of fibres of $\kappa$ which allows us to compute $Y(\Z_S)$ from $\kappa(Y(\Z_S))$, and the modularity of elliptic curves provides a constructive finiteness of $\kappa(Y(\Z_S))$. In order to show that our approach to compute $Y(\Z_S)$ works in practice, we append the tables of complete solutions for the following equations.
\begin{align}\label{eq:1.5}
\begin{array}{|c|c|c|}
\hline 
h(x,y)		& S 		& \text{  Table  }
\\	\hline
x(x-y)y		& \{2,7,11,13\}		& \text{\ref{01-10271113}}
\\	\hline
x(x-y)y		& \{2,3,431\}		& \text{\ref{01-1023431}}
\\	\hline
x(x-y)y		& \{2,3,5,53\}		& \text{\ref{01-1023553}}
\\	\hline
(x^2+7y^2)y	& \{2,3,5,7\}		& \text{\ref{01072357}}
\\	\hline
(x^2+7y^2)y 	& \{2,7,11, 13\}		& \text{\ref{0107271113}}
\\	\hline
(x^2+3y^2)y 	& \{2,3,11\}		& \text{\ref{01022311}}
\\	\hline2
(x^2+y^2)y 	& \{2,3,7,11\}		& \text{\ref{010123711}}
\\	\hline
(x^2+y^2)y 	& \{2,5,13\}		& \text{\ref{01012513}}
\\	\hline
(x^2-2y^2)y	& \{2,7,29\} 		& \text{\ref{010-2251317}}
\\	\hline
(x^2-2y^2)y	& \{2,7,29\} 		& \text{\ref{010-22729}}
\\	\hline
(x^2-3y^2)y	& \{2,5,7,11\} 		& \text{\ref{010-323711}}
\\	\hline
(x^2-7y^2)y	& \{2,3,7,11\} 		& \text{\ref{010-725711}}
\\	\hline
x^3-x^2y-4xy^2-y^3 &\{2,5,13\}		& \text{\ref{1-1-4-12513}}
\\	\hline
x^3-x^2y-2xy^2-2y^3 &\{2,5,19\}	& \text{\ref{1-1-2-22519}}
\\	\hline
x^3+y^3 		&\{2,3,5\}			& \text{\ref{1001235}}
\\	\hline
x^3+2y^3 		&\{2,3,5\}			& \text{\ref{1002235}}
\\	\hline
x^3-y^3 		&\{2,3,5\}			& \text{\ref{100-2235}}
\\	\hline
x^3-2y^3 		&\{2,3,5\}			& \text{\ref{100-3235}}
\\	\hline
\end{array}
\end{align}
\par
Note that we omitted the trivial solution $h(1,0)=1$ in the appended tables, and Table~\ref{01-1023553} had to be abbreviated due to a large number of solutions.
\par
The choices of $h(x,y)$ and $S$ in \eqref{eq:1.5} had to be restricted according to our computational capability, and the choices are made to show the flexibility that we have. The main restriction comes from one's ability to find the $c_4$ and $c_6$ invariants of all elliptic curves whose conductor divides the worst possible conductor given in Proposition~\ref{prop:2.1}. Other computational difficulties are negligible. Computational issues are further discussed in Section~\ref{section:7}.
\par
When we computed the solutions of the equations listed above, we exploited Cremona's Elliptic Curve Database, from which we read the coefficients of elliptic curves with specified conductor. After that, we compute $\kappa^{-1}(E)$ for each curve $E$ read from the database. Of course, it is highly non-trivial task to establish a complete list of isomorphism classes of elliptic curves of specified conductor, and we are outsourcing this job to Cremona. We note that this job is computational infeasible without modularity, and even with modularity it takes significant further efforts to obtain practically efficient algorithm. Nevertheless, the absence of modularity is the main theoretical and practical obstacle to generalising our method to number fields. Once we have the coefficients of necessary elliptic curves, then it takes no more then $20$ seconds to generate each of the tables listed above.
\par
The spectacular resolution of Fermat's Last Theorem by means of modular methods as well as \cite{Bennett Skinner}, \cite{Bugeaud Mignotte Siksek I}, \cite{Bugeaud Mignotte Siksek II} and \cite{Darmon Merel} uses level lowering argument in a crucial way, which allows one to produce obstruction to existence of a solution by numerically showing that certain space of modular forms is zero dimensional. In contrast, we will be using modularity theorem in order to produce a complete set of solutions for a given Thue-Mahler equation, without a priori guess on the number of solutions. In particular, we do not use any form of level lowering.
\par
We outline the contents of the paper. In Section~\ref{section:2}, we define $\kappa$ and study basic properties of the elliptic curve $X_t$ associated to $t \in Y(\Z_S)$. In Section~\ref{section:3}, we study $\kappa^{-1}(E)$ for an elliptic curve $E$ given in terms of a Weierstrass equation. In Section~\ref{section:4}, we prove the main theorem on the upper bound of the cardinality of $Y(\Z_S)$, compare our upper bound with Evertse's upper bound. The proof is manifestly constructive, and it provides an algorithm to determine $Y(\Z_S)$. We discuss the algorithmic aspect in Section~\ref{section:5}. In Section~\ref{section:6}, we explain how the algorithm is implemented in the computer algebra package, and discuss its performance. We also list the cardinalities of $Y(\Z_S)$ as we vary $S$. In Section~\ref{section:7}, we recall the work of Tzanakis and de Weger which proposed, based on Baker's method and further optimisations, a practical algorithm for Thue-Mahler equations, and we compare the natures of two approaches. In Section~\ref{section:8}, we specialise $h(x,y)$ in order to explain the connection to generalised Ramanujan-Nagell equations. Some particular generalised Ramanujan-Nagell equations are solved, from which we observe a pattern among the number of solutions as we vary $S$.

\par
\section{Definition of $\kappa$ and its properties}\label{section:2}
Consider a binary cubic form
\begin{align}
h(x,y) = ax^3 + bx^2y + c xy^2 + d y^3
\end{align}
with relatively prime integer coefficients. We assume that the discriminant 
\begin{align}
\delta = 3b^2c^2 + 6 abcd - 4b^3d - 4ac^3 - a^2d^2
\end{align}
of $h(x,y)$ is non-zero, or equivalently that $h(x,y)$ has three projectively non-equivalent zeros over an extension of $\Q$.
\par
Let $\mb P_{xy}^1$ be the projective line with homogeneous coordinates $x$ and $y$. Let $Z$ be the subscheme of $\mb P_{xy}^1$ defined by $h(x,y)=0$, and let $Y$ to be the complement
\begin{align}
Y = \mathbb P^1 - Z
\end{align}
which we view as an affine variety embedded in $\mathbb P_{xy}^1$. In particular, a point $t$ in $Y(R)$ for some ring $R$ will be represented as a pair $(x_t:y_t) \in \mb P^1(R)$ such that $h(x_t,y_t)$ is a unit in $R$.
\par
The aim of this section is to introduce the map $\kappa$ which associates an elliptic curve to a point in $Y(R)$, and to study its basic properties. We will first construct a generically smooth map $X \to Y$, and the associated elliptic curve will be constructed by taking its fibre. The associated elliptic curve is naturally equipped with additional structures, which we will analyse in this section. 

\subsection{Coordinates of $Y$}
Our definition of $Y$ as an open subscheme of $\mathbb P^1_{xy}$, endows $Y$ with homogeneous coordinates $x$ and $y$, but we would like to introduce another coordinate $\epsilon$ of $Y$, which makes our later discussion simpler. Consider $\widetilde Y\subset \mathbb A^3_{xy\epsilon}$, defined by
\begin{align}
\label{eq:3.6}
\widetilde Y \colon h(x,y)\epsilon = 1
\end{align}
where $h(x,y)$ is the defining equation of $Z = \mathbb P^1-Y$. Let $\mathbb G_m$ act on $\mathbb A^3_{xy\epsilon}$ with weight $1$, $1$, and $-3$. That is to say, for any ring $R$, $\lambda \in R^\times$, and $(x,y,\epsilon) \in \mathbb A^3_{xy\epsilon}(R)$, the action of $\lambda$ is given by
\begin{align}
\lambda \cdot (x,y,\epsilon) = \left ( \lambda x, \lambda y, \lambda^{-3}\epsilon \right).
\end{align}
Because the action preserves \eqref{eq:3.6}, one can consider the quotient $\mathbb G_m \backslash \widetilde Y$, which is just $Y$. Indeed, $x$ and $y$ are homogeneous coordinates of degree one, defining the embedding $\mathbb G_m \backslash \widetilde Y \to \mathbb P^1$. The coordinate $\epsilon$ of $Y$ is redundant, but it will be convenient for later purposes.

\subsection{Construction of the family $f \colon X \to Y$.}
We describe $f \colon X \to Y$ in this subsection. We will define $X$ as a quotient of $\widetilde X$, where $\widetilde X$ is an affine subscheme of $\mathbb A^3_{xy\epsilon} \times \mathbb A^3_{uvw}$. The defining equations of $\widetilde X$ are
\begin{align}
\label{eq:3.8}
h(x,y)\epsilon &= 1
\\
\label{eq:3.9}
\epsilon \cdot w^2 & = h(u,v) (yu - xv).
\end{align}
Now we let $\mathbb G_m \times \mathbb G_m$ act on $\widetilde X$ in the following way. If $R$ is a ring, $\left(\lambda,\mu\right) \in \mathbb G_m(R) \times \mathbb G_m(R)$, and $(x,y,\epsilon,u,v,w) \in \widetilde X$, then we define
\begin{align}
\label{eq:3.10}
\left(\lambda,\mu\right) \cdot (x,y,\epsilon,u,v,w) = (\lambda x, \lambda y, \lambda ^{-3}\epsilon, \mu u, \mu v,\lambda^2 \mu^2w).
\end{align}
Since \eqref{eq:3.8} and \eqref{eq:3.9} are preserved by the action \eqref{eq:3.10}, we may define 
\begin{align}
X = \mathbb G_m \times \mathbb G_m \backslash \widetilde X.
\end{align}
Furthermore, the projection $\widetilde X \to \widetilde Y$ descends to $X \to Y$, which we denote by $f$.
\begin{remark}
Geometrically speaking, $X$ parametrises the double covers of $\mathbb P^1$ branched along the divisor $Z \cup \{ t \}$ of degree four, as $t$ varies in $Y$. However, this does not uniquely characterise $X$, since there are more than one such double covers which are not isomorphic to each other over $\Q$.
\end{remark}
\begin{remark}
When $h(x,y) = x(x-y)y$, $Y$ can be identified with the affine line without $0$ and $1$, by taking $\lambda = x/y$ as an affine coordinate. Then, the equation
\begin{align}
w^2 = u(u-v)v(u-\lambda v)
\end{align} 
defines the Legendre family of elliptic curves over $Y$. Our family $X\to Y$ of elliptic curves in this case is represented by
\begin{align}\label{eq:2.11}
(x(x-y)y^2)^{-1}w^2 = u(u-v)v(u-\lambda v)
\end{align}
so one can view it as a quadratic twist of the original Legendre family by $x(x-y)y^2$. This relation between Legendre family and our $X$ is available because of the affine coordinate $\lambda$ for $Y$. When $h(x,y)$ has no rational linear factor, such an affine coordinate is not available, whence the original Legendre family does not directly generalise. It is the advantage of the twisted family \eqref{eq:2.11} that it generalises to general $h(x,y)$.
\end{remark}

\subsection{Properties of $f\colon X\to Y$}\label{subsection:2.3}
In this subsection, we study basic properties of $f \colon X \to Y$. For some ring $R$ and $t = (t_x:t_y) \in Y(R)$, the elliptic curve $X_t$ is defined by the equation
\begin{align}\label{eq:2.12}
X_t : \epsilon \cdot w^2 & = h(u,v) (y_tu - x_tv)
\end{align}
which be interpreted as a quadratic twist of
\begin{align}\label{eq:2.14}
X_t' : w^2 & = h(u,v) (y_tu - x_tv)
\end{align}
by $\epsilon$, although $X_t'$ does not patch together to form a family of elliptic curves over $Y$. In fact, $X_t'$ is not even well-defined on the projective equivalence class of $t=(x_t:y_t)$, and only its quadratic twist $X_t$ is well-defined. In any case, we can compute the discriminant of the right hand side of \eqref{eq:2.14}, as well as prove its properties.
\begin{proposition}\label{prop:2.1}
The discriminant of right hand side of \eqref{eq:2.14} is
\begin{align}\label{eq:2.15}
h(x_t,y_t)^2\cdot \delta
\end{align}
where $\delta$ is the discriminant of $h(x,y)$. Furthermore, we have:
\begin{enumerate}
\item if $t \in Y(\Z_S)$, then $X_t'$ and has good reduction outside of $S$ and $2\delta$. 
\item if an odd prime $p \in S$ is a prime of bad reduction for $X_t'$ and $p$ does not divide $2\delta$, then $h(x,y)=0$ has at least one solution modulo $p$.
\item if $p$ does not divide $2\delta$, then $X_t'$ has either good or multiplicative reduction.
\end{enumerate}
\end{proposition}
\begin{proof}
The formula \eqref{eq:2.15} follows from representation of the discriminant in terms of differences of roots. Indeed, if $P(x)$ is a polynomial of one variable with roots $\alpha_1,\cdots, \alpha_n$, then the discriminant $\delta_P$ of $P(x)$ is
\begin{align}
\prod_{i<j}(\alpha_i-\alpha_j)^2.
\end{align}
If $P(x) = (x-\beta)Q(x)$, and $\alpha_n=\beta$, then the above formula can be rewritten
\begin{align}
\prod_{i<n}( \beta-\alpha_i )^2 \times \prod_{i<j<n}(\alpha_i-\alpha_j)^2
\end{align}
which equals $P_1(\beta)^2 \cdot \delta_Q$.
\par
If $t \in Y(\Z_S)$, and $(x_t:y_t)$ is some representative of $t$, then $X_t'$ has good reduction away from $S$ and $2\delta$. Indeed, a double cover of $\mb P^1$ branched along four distinct points is smooth away from characteristic two.
\par
Suppose $p \in S$ does not divide $2\delta$ and $X_t'$ has bad reduction at $p$. Since $p$ does not divide $\delta$, $h(x,y)$ has three distinct roots modulo $p$. Thus $X_t'$ has bad reduction if and only if $t$ coincides with one of three zeroes of $h(x,y)$. In particular, $t$ is a solution of $h(x,y)=0$ modulo $p$.
\par
Assume that $p$ does not divide $2\delta$ and $X_t'$ has a bad reduction at $p$. Then the right hand side of \eqref{eq:2.14} cannot have a cubic factor, since such a factor will force $h(x,y)$ to have at least a square factor modulo $p$, contradicting the assumption. In other words, $X_t'$ has either good or multiplicative reduction. This completes the proof of the proposition.
\end{proof}
For $X_t$, we can prove the following:-
\begin{proposition} With the notation as in the previous proposition, the discriminant of \eqref{eq:2.12} is
\begin{align}
h(x_t,y_t)^4\delta.
\end{align}
In particular, for any $t \in Y(\Z_S)$, $X_t$ has good reduction outside of $S$ and $2\delta$.
\end{proposition}
\begin{proof}
If follows immediately from Proposition~\ref{prop:2.1}, and the description of the discriminant in terms of differences of roots.
\end{proof}
\begin{remark}
The dependence of $X_t'$ on the choice of representative $(x_t:y_t)$ is not so serious, as far as we work with rational numbers. We can always take $(x_t:y_t)$, such that $x_t$ and $y_t$ are relatively prime integers and $x_t$ is non-negative. If we work over a number field which has either positive class number or more units than $\pm 1$, this is not straightforward. Working with $X_t'$ has advantage that the conductor of $X_t'$ is usually smaller than $X_t$ and it has in some sense finer information about $t$ than $X_t$ does. On the other hand, $X_t$ is associated naturally from $t$, whose isomorphism class is independent of choice of representative for $t$, so it is technically more convenient.
\end{remark}

\section{Defining equations of fibres of $\kappa$}\label{section:3}
In the previous section, we defined a map
\begin{align}
\kappa \colon t \mapsto X_t
\end{align}
which associates an elliptic curve $X_t$ to a solution $t \in Y(\Z_S)$. The aim of the present section is to describe $\kappa^{-1}(E)$ as a variety defined by explicit polynomials.
\subsection{Some invariant theory}
We briefly review basic invariant theory of binary forms which are relevant for us. We start with invariants of binary quartic forms.
\par
Let
\begin{align}
q = A_0u^4 + A_1uv^3 + A_2u^2v^2 + A_3 uv^3 + A_4v^4
\end{align}
be a generic binary quartic form in $u,v$, with coefficients $A_i$'s. We choose
\begin{align}
I_2 &= \frac{1}{12}A_2^2 - \frac 1 4 A_1A_3 + A_0A_4
\\
I_3 & = \frac {1}{216}A_2^3 - \frac{1}{48}A_1A_2A_3 + \frac{1}{16}A_0A_3^2 + \frac{1}{16}A_1^2A_4 - \frac{1}{6}A_0A_2A_4
\end{align}
as generators for the ring of invariants of binary quartic forms. Note that they have rational coefficients, and their degrees are two and three respectively.
\par
In fact, these two invariants are algebraically independent, or the ring of invariants is isomorphic to the polynomial ring in two variables. To stress the dependence of the invariants on the coefficients, we denote by $q(A)$ the quartic form with coefficients $A=(A_0,A_1,A_2,A_3,A_4)$, and denote their invariants by $I_2(A)$ and $I_3(A)$ respectively.
\begin{proposition}
Let $q(A)$ and $q(A')$ are two binary cubic forms with rational coefficients. They are linearly equivalent over $\Q$ if and only there exists $\lambda \in \Q^\times$ such that
\begin{align}
I_2(A) &= \lambda^2 (A')
\\
I_3(A) &= \lambda^3(A')
\end{align}
holds.
\end{proposition}
\begin{proof}
Classical invariant theory.
\end{proof}

Now we consider the invariant theory of a pair of binary forms. Let
\begin{align}
q(A,B) = (A_0u + A_1v)(B_0u^3 + B_1 u^2 v + B_2 uv^2 + B_3 v^3)
\end{align}
be a product of a binary linear form and a binary cubic form, where $A=(A_0,A_1)$ and $B=(B_0,\cdots,B_3)$ denote the coefficients of a linear form and a cubic form respectively. Two invariants
\begin{align}
\label{eq:3.7}
c_4(A,B) =&-16   (-A_1^2 B_1^2 + 3 A_1^2 B_0 B_2 + A_0 A_1 B_1 B_2 - A_0^2 B_2^2 - 9 A_0 A_1 B_0 B_3 + 3 A_0^2 B_1 B_3)
\\
c_6(A,B) =&-32   (2 A_1^3 B_1^3 - 9 A_1^3 B_0 B_1 B_2 - 3 A_0 A_1^2 B_1^2 B_2 + 18 A_0 A_1^2 B_0 B_2^2 - 3 A_0^2 A_1 B_1 B_2^2 
\\&
+ 2 A_0^3 B_2^3 + 27 A_1^3 B_0^2 B_3 - 27 A_0 A_1^2 B_0 B_1 B_3 + 18 A_0^2 A_1 B_1^2 B_3 - 27 A_0^2 A_1 B_0 B_2 B_3 
\\&
- 9 A_0^3 B_1 B_2 B_3 + 27 A_0^3 B_0 B_3^2)
\end{align}
generate the ring of invariants.
\par
We digress for a discussion on the relation between above invariants and invariants of an elliptic curve often used in the literature. If an elliptic curve $E$ is given by
\begin{align}\label{eq:3.11}
E: y^2 = x^3 + a_2x^2 + a_4x + a_6
\end{align}
then
\begin{align}\label{eq:3.13}
c_4(E) =& 16a_2^2 - 48 a_4
\\\label{eq:3.14}
c_6(E) =& -64 a_2^3 + 288 a_2 a_4 - 864 a_6
\end{align}
are often called $c_4$-invariant and $c_6$-invariant of $E$. If we take $A_0=0$, $A_1=1$, $B_0=1$, $B_1=a_2$, $B_2=a_4$, and $B_4=a_6$, then $c_4(A,B)$ and $c_6(A,B)$ are precisely the $c_4$-invariant and $c_6$-invariant of the elliptic curve defined by \eqref{eq:3.11}. The discriminant of \eqref{eq:3.11} is given by
\begin{align}
\delta(E) = -16 (-a_2^2a_4^2 + 4 a_2^3 a_6 + 4 a_4^3 - 18 a_2 a_4 a_6 + 27 a_6^2)
\end{align}
and it can be alternatively written as
\begin{align}
\delta (E) = 2^63^3(c_4(E)^3-c_6(E)^2)= 1728(c_4(E)^3-c_6(E)^2).
\end{align}
The discriminant of $q(A,B)$ viewed as a single binary quartic form, its discriminant is
\begin{align}
\delta(q(A,B)) =& -(-B_1^2 B_2^2 + 4 B_0 B_2^3 + 4 B_1^3 B_3 - 18 B_0 B_1 B_2 B_3 + 27 B_0^2 B_3^2)  
\\&
\times (-A_1^3 B_0 + A_0 A_1^2 B_1 - A_0^2 A_1 B_2 + A_0^3 B_3)^2.
\end{align}
If we take $A_0=0$, $A_1=1$, $B_0=1$, $B_1=a_2$, $B_2=a_4$, and $B_4=a_6$, then two discriminants are related by
\begin{align}\label{eq:3.18}
16\cdot \delta(q(A,B)) =\delta(E).
\end{align}
\par
The relation between $c_4$, $c_6$ and the previously introduced $I_2$ and $I_3$ are more straightforward. Indeed, they are related by
\begin{align}
c_4(q(A,B)) = & 192 \cdot I_2 (q(A,B))
\\
= & 2^6 3 \cdot I_2(q(A,B))
\\
c_6(q(A,B)) = & -13824 \cdot I_3(q(A,B))
\\
=& -2^9 3^3 \cdot I_3(q(A,B))
\end{align}
where we view $q(A,B)$ as a product of two forms on the left hand side, while on the right hand side we view it as a single quartic form whose coefficients are quadratic forms in $A_i$'s and $B_i$'s.
\par
Now we return to the invariant theory of $q(A,B)$.
\begin{proposition}
Let $q(A,B)$ and $q(A',B')$ are two binary quartic forms with factorisation as a product of a linear and a cubic factor, and suppose that the coefficients $A,A',B$ and $B'$ are rational numbers. They are linearly equivalent over the rational numbers if and only if there exists $\lambda \in \Q^\times$ such that
\begin{align}
c_4(q(A,B))=&\lambda^4 c_4(q(A',B'))
\\
c_6(q(A,B))=&\lambda^6 c_6(q(A',B'))
\end{align}
holds.
\end{proposition}
\begin{proof}
Classical invariant theory.
\end{proof}

\subsection{Faithfulness of descent}
The overall strategy is to study $t \in Y(\Z_S)$ in terms of $X_t$. In other words, we consider the map
\begin{align}
\kappa \colon t \mapsto X_t
\end{align}
and try to use $\kappa$ in order to compute $Y(\Z_S)$. To realise this, we will show two key properties of $\kappa$:
\begin{enumerate}
\item $\kappa$ is $n$-to-$1$ map, where $n$ is an explicit integer less than six.
\item Given a Weierstrass equation of an elliptic curve $E$, one can compute $\kappa^{-1}(E)$.
\end{enumerate}
Let us consider the first property of $\kappa$. Let $E$ be an elliptic curve. We would like to count the number of $t \in Y(\Z_S)$ for which $X_t $ is isomorphic to $E$, where $E$ is given as a Weierstrass equation
\begin{align}
E : y^2 +a_1 y + a_3= x^3 + a_2 x^2 + a_4 x + a_6
\end{align}
with rational coefficients. On the other hand, $X_t$ is defined by 
\begin{align}
\epsilon w^2 = h(u,v) (yu - xv)
\end{align}
where $h(x,y)\epsilon =1$. We rewrite $X_t$ as
\begin{align}
w^2 = h(u,v)(yu-xv)h(x,y)
\end{align}
and let
\begin{align}
Q(x,y,h) = (yu-xv) \cdot h(u,v)h(x,y)
\end{align}
be the associated joint quartic form in $u$ and $v$. In fact, one could write $Q(x,y,h)=Q(t,h)$, in the sense that $Q(\lambda x, \lambda y, h)$ is rationally equivalent to $Q(x,y,h)$ for any $\lambda \in \Q^\times$. For simplicity of notation, let
\begin{align}
c_4(x,y,h) = c_4(Q(x,y,h))
\\
c_6(x,y,h) = c_6(Q(x,y,h))
\end{align}
be the invariants of $Q(x,y,h)$. The set of all $t \in Y(\Z_S)$ for which $X_t$ is ismorphic to $E$ is defined by the equations
\begin{align}
c_4 (x,y,h) = \lambda^4 c_4(E)\label{eq:3.32}
\\
c_6(x,y,h) = \lambda^6 c_6(E)\label{eq:3.33}
\end{align}
where $c_4(x,y,h)$ and $c_6(x,y,h)$ are homogeneous polynomials of degree eight and twelve in $x,y$, respectively. In fact, we can eliminate $\lambda$ from above two equations to obtain
\begin{align}\label{eq:3.34}
J_{24}(x,y,h,E) := \lambda^{12}\left( c_6(E)^2c_4(x,y,h)^3 - c_4(E)^3 c_6(x,y,h)^2 \right)
\end{align}
which factors as
\begin{align}\label{eq:3.35}
J_{24}(x,y,h,E) =&\lambda^{12}\cdot h(x,y)^{6}\cdot J_{6}(x,y,h,E)
\end{align}
where $J_{6}(x,y,h,E)$ is the homogeneous polynomial of degree six in variables $x$ and $y$, characterised by above equality.
\begin{proposition}\label{prop:3.3}
The set of points $t \in Y(\Z_S)$ for which $X_t$ is isomorphic to $E$ is in bijection with the projective equivalence class of solutions of $J_6(x,y,h,E)=0$ such that \eqref{eq:3.32} and \eqref{eq:3.33} have a common solution in $\lambda$. In particular, this set has cardinality at most six.
\end{proposition}
\begin{proof}
We first show that $J_6(x,y,h,E)$, viewed as a homogeneous polynomial in $x$ and $y$, is not identically zero for an elliptic curve $E$ and a non-degenerate binary cubic form $h(x,y)$. We work with complex numbers, although any algebraically closed field of characteristic zero suffices our purpose. Let $(x_i,y_i)$ with $i=1,2,\cdots,r$ be a sequence of non-equivalent complex zeros of $J_6(x,y,h,E)$. They are precisely the values for which
\begin{align}
Q_i(x_i,y_i,h) = (y_iu-x_iv) \cdot h(u,v)h(x_i,y_i)
\end{align}
becomes equivalent to
\begin{align}
Q_E := v(u^3 + a_4(E) uv^2 + a_6(E)v^3)
\end{align}
as joint binary quartic forms in $u$ and $v$. Isomorphisms between two joint forms $Q_i$ and $Q_E$ correspond to the projective isomorphisms between two divisors represented as zeros of $Q_i$ and $Q_E$, which sends $(x_i,y_i)$ to $(0,1)$. Since $E$ is an elliptic curve, divisor of $u^3 + a_4(E) uv^2 + a_6(E)v^3$ consists of three distinct points. There are six projective automorphisms of $\mb P^1_{xy}$ which sends the divisor of $h(u,v)$ to the divisor of $u^3 + a_4(E) uv^2 + a_6(E)v^3$. Thus $r$ is at most six, and $J_6(x,y,h,E)$ cannot be identically zero.
\par
If $t=(x:y) \in Y(\Z_S)$, then $h(x,y) \not =0$. Thus the vanishing of $J_{24}(x,y,h,E)$ is equivalent to vanishing of $J_6(x,y,h,E)$, from factorisation of $J_{24}(x,y,h,E)$. If $t$ further satisfies the condition $X_t $ being isomorphic to $E$, then clearly \eqref{eq:3.32} and \eqref{eq:3.33} have a common solution. Conversely, if $J_6(x,y,h,E)$ vanishes at some point $t$, then for each solution satisfying \eqref{eq:3.32} and \eqref{eq:3.33}. Thus such $t$ together with $\lambda$ gives rise to $X_t $ equipped with an isomorphism to $E$. This completes the proof.
\par
\end{proof}
\begin{corollary}\label{cor:3.1}
For a fixed elliptic curve $E$, the number of $t \in Y(\Z_S)$ for which $X_t $ is isomorphic to $E$ is at most the cardinality of $\mathrm{Aut}_\Q(Y)$.
\end{corollary}
\begin{proof}
It is implicit in the proof of Proposition~\ref{prop:3.3}, observing that the isomorphisms between $h(u,v)$ and $u^3 + a_4(E) uv^2 + a_6(E)v^3$ is a torsor for $\mathrm{Aut}_\Q(Y)$.
\end{proof}
Before we move on, we analyse coefficients of $J_6(x,y,h,E)$. As we mentioned earlier, it is homogeneous of degree six in $x$ and $y$. With respect to the coefficients of $h(x,y)$, namely $a,b,c,$ and $d$, it is homogeneous of degree six as well. In terms of coefficients of $E$, it has degree twelve in the following sense. If we write $E$ as 
\begin{align}
y^2 + a_1 xy + a_3 y = x^ 3 + a_2 x^2 + a_4 x + a_6
\end{align}
then $J_6(x,y,h,E)$ is a polynomial in variables $a_1,a_2,a_3,a_4$ and $a_6$. If we take degree of $a_m$ to be $m$, then $J_6(x,y,h,E)$ is homogeneous of degree twelve in $a_m$'s. More concretely, if we take a model of $E$ for which $a_1=a_2=a_3=0$, then each term of $J_6(x,y,h,E)$, viewed as a polynomial in $a_4$ and $a_6$, is either $a_4^3$ or $a_6^2$. Based on this observation, we arrange
\begin{align} 
J_6(x,y,h,E) = \sum_{i = 0}^6 \left(C^*_i(h)a_4^3 + D^*_i(h)a_6^2 \right)x^{6-i}y^i
\end{align}
where $C^*_i(h)$ and $D*_i(h)$ are homogeneous polynomials of degree six in $a,b,c,$ and $d$. In fact, all of them have a large integer factor, so we let
\begin{align}
C_i^*(h) = 2^{22}3^3C_i(h) \text{ and } D_i^*(h) = 2^{22}3^3 D_i(h)
\end{align}
and give formulas for $C_i(h)$ and $D_i(h)$:
\begin{align*}
C_0(h)&=
(2  c^3 - 9  b  c  d + 27  a  d^2)^2
\\
D_0(h)&=
-27    (-c^2 + 3  b  d)^3
\\
C_1(h)&=
-6    (2  c^3 - 9  b  c  d + 27  a  d^2)    (-b  c^2 + 6  b^2  d - 9  a  c  d)
\\
D_1(h)&=
-81    (-b  c + 9  a  d)    (-c^2 + 3  b  d)^2
\\
C_2(h)&=
-3    (b^2  c^4 - 24  a  c^5 + 18  b^3  c^2  d + 90  a  b  c^3  d - 108  b^4  d^2 + 216  a  b^2  c  d^2 - 567  a^2  c^2  d^2 + 486  a^2  b  d^3)
\\
D_2(h)&=
-81    (-c^2 + 3  b  d)    (2  b^2  c^2 - 3  a  c^3 - 3  b^3  d - 9  a  b  c  d + 81  a^2  d^2)
\\
C_3(h)&=
-2    (13  b^3  c^3 - 72  a  b  c^4 - 72  b^4  c  d + 567  a  b^2  c^2  d - 432  a^2  c^3  d - 432  a  b^3  d^2 + 243  a^2  b  c  d^2 + 729  a^3  d^3)
\\
D_3(h)&=
-27    (-b  c + 9  a  d)    (7  b^2  c^2 - 18  a  c^3 - 18  b^3  d + 36  a  b  c  d + 81  a^2  d^2)
\\
C_4(h)&=
-3    (b^4  c^2 + 18  a  b^2  c^3 - 108  a^2  c^4 - 24  b^5  d + 90  a  b^3  c  d + 216  a^2  b  c^2  d - 567  a^2  b^2  d^2 + 486  a^3  c  d^2)
\\
D_4(h)&=
-81    (-b^2 + 3  a  c)    (2  b^2  c^2 - 3  a  c^3 - 3  b^3  d - 9  a  b  c  d + 81  a^2  d^2)
\\
C_5(h)&=
6    (b^2  c - 6  a  c^2 + 9  a  b  d)    (2  b^3 - 9  a  b  c + 27  a^2  d)
\\
D_5(h)&=
-81    (-b  c + 9  a  d)    (-b^2 + 3  a  c)^2
\\
C_6(h)&=
(2  b^3 - 9  a  b  c + 27  a^2  d)^2
\\
D_6(h)&=
-27    (-b^2 + 3  a  c)^3
\end{align*}
Above coefficients seem to admit certain invariant theoretic meanings, but for the purpose of the current article, we only need to regard $J_6(x,y,h,E)$ as a defining equation of $\kappa^{-1}(E)$ given as explicitly as possible.  
\section{Main theorems}\label{section:4}
Now we are ready to state and prove main theorems. Before we state our result, recall that the strategy was to consider
\begin{align}
\kappa \colon t \mapsto X_t
\end{align}
and study $t$ in terms of $X_t$. In the previous subsection, we focused on the analysis of $\kappa^{-1}(E)$ for a fixed $E$. What remains is to consider the image of $\kappa$.
\par
Among several approaches to analyse the image of $\kappa$, the common fundamental property of $\kappa$ is that $X_t$ has good reduction outside from $2\delta$ and $S$. It is convenient to introduce a notation for the number of such curves.
\begin{definition}
Let $N$ be a positive integer. Let $g(N)$ be the number of isomorphism classes of elliptic curves whose conductor is $N$. For a finite set $S$ of primes, let
\begin{align}
M = \prod_{p \in S} p
\end{align}
and let $G(S)$ be the number of isomorphism classes of elliptic curves which have good reduction outside of $S$.
\end{definition}

\begin{theorem}\label{thm:3.1}
Let $\varepsilon$ be any positive integer. There exists constants $k_1$, $k_2$ depending on $\varepsilon$ such that
\begin{align}
g(N) < k_1 N^{\frac 1 2 + \varepsilon}
\\
G(S) < k_2 M^{\frac 1 2 + \varepsilon}
\end{align}
holds. If we assume (a part of) BSD conjecture and generalised Riemann hypothesis for the elliptic curves $y^2 = x^3 + n$, $n \in \Z$, then there exists constants $k_4$ and $k_5$ for which
\begin{align}
G(S) < k_4 M^{\frac{k_5}{\log \log M}}
\end{align}
\end{theorem}
\begin{proof}
It is a theorem of Brumer and Silverman.
\end{proof}
As a corollary of above theorem, we obtain a first, rather crude, upper bound for the cardinality of $Y(\Z_S)$.
\begin{theorem}\label{thm:3.2}
Assume that $S$ contains $2$ and prime divisors of $\delta$. Then
\begin{align}
\left |Y(\Z_S) \right |< \left |\mathrm{Aut}_\Q(Y) \right | \times G(S).
\end{align}
In particular, 
\begin{align}
\left |Y(\Z_S) \right |< \left |\mathrm{Aut}_\Q(Y) \right |\times k_2 M^{\frac{1}{2}+\varepsilon}.
\end{align}
where $M$ is the product of all primes numbers in $S$, and $k_2$ is a constant which only depends on an arbitrary positive number in Theorem~\ref{thm:3.1}. Conditional on generalised Riemman hypothesis and BSD conjecture for the curve $y^2 = x^3 + n$, we have
\begin{align}
\left |Y(\Z_S) \right |< \left |\mathrm{Aut}_\Q(Y) \right |\times k_4 M^{\frac{k_5}{\log\log M}}.
\end{align}
where $k_4$ and $k_5$ are absolute constants.
\end{theorem}
\begin{proof}
The first assertion follows from Proposition~\ref{prop:3.3} and Corollary~\ref{cor:3.1}. The second follows from the first assertion combined with Theorem~\ref{thm:3.1}.
\end{proof}
\begin{remark}
In the course of deriving the above upper bound, we have forgotten all information of $X_t$ except the divisors of the conductors of $X_t$. Numerical computations indicate that the reduction type at the primes dividing $2\delta$ takes a particular form, which will facilitate practical computations.
\end{remark}
\par

\subsection{Comparison with Evertse's bound}

Evertse proved the following remarkable upper bound:-
\begin{theorem}
Let $h(x,y)$ be an integral binary form of degree $n \ge 3$ which is divisible by at least three pairwise linearly independent linear forms in some algebraic number field. Let $p_1, p_2, \cdots, p_t$ be a sequence of distinct primes. The equation
\begin{align}
|h(x,y)|= \prod_{i = 1}^tp_i^{e_i}
\end{align}
in relatively prime integers $x,y$ and non-negative integers $k_1,k_2,\cdots,k_t$ has at most
\begin{align}
2\times7^{n^3(2t+3)}
\end{align}
solutions. There is analogous explicit upper bound for number fields.
\end{theorem}
\begin{proof}
Corollary~2 of \cite{Evertse}
\end{proof}
As a direct consequence, we obtain
\begin{align}
\left| Y(\Z_S) \right | < 7^{9(s+3)}
\end{align}
when $s$ is the cardinality of $S$. Let us take $S$ to be the set
\begin{align}
S = \{p: p < T\}
\end{align}
of all primes up to a positive number $T$. Then by prime number theorem $s$ is asymptotically $T / \log T$. Therefore, Evertse's bound can be rewritten, in a logarithmic scale, as
\begin{align}
\log  \left| Y(\Z_S) \right | = O\left(\frac{T}{\log T}\right).
\end{align}
On the other hand, the standard estimate shows that
\begin{align}
\log M = \sum_{p < T}\log p = O(T)
\end{align}
and our unconditional upper bound of Theorem~\ref{thm:3.1} becomes
\begin{align}
\log  \left| Y(\Z_S) \right | = O\left({T}\right).
\end{align}
The conditional upper bound of Theorem~\ref{thm:3.1} becomes
\begin{align}
\log  \left| Y(\Z_S) \right |  = O\left(\frac{\log M}{\log\log M}\right)= O\left(\frac{T}{\log T}\right)
\end{align}
which is comparable to Evertse's.

\par
Putting aside the comparisons between upper bounds for the cardinalities $Y(\Z_S)$, we point out crucial difference of our method from Evertse's in terms of effectiveness. Evertse's upper bound is ineffective, in the sense that it does not provide an algorithm to decide $Y(\Z_S)$. In contrast, our proof is manifestly constructive, especially if one combines with modularity of elliptic curves. We elaborate on the constructive aspects of our method in the next section.

\section{Algorithmic aspects: effective Mordell conjecture}\label{section:5}
In this subsection, we elaborate on the algorithmic aspects of our proof, which answers the effective Mordell conjecture for $Y$. Effective finiteness of $Y(\Z_S)$ can be formulated in at least two ways:
\begin{enumerate}
\item to have an explicit upper bound on height of $t \in Y(\Z_S)$ in terms of coefficients of $h(x,y)$ and $S$.
\item to have a procedure which enables one to determine $Y(\Z_S)$, in provably finite amount of time, for a numerically given $h(x,y)$ and $S$.
\end{enumerate}
\par
The first version of effectiveness implies the second. Indeed, if $T$ is obtained bound, then factoring $h(m,n)$ as $m$ and $n$ varies among all integers with absolute value at most $T$, one can determine $Y(\Z_S)$. In fact, Baker's method in principle provides such an upper bound. However, the efficiency of such procedure depends on the size of $T$, and the astronomical size of numbers obtained from Baker's bound often prevents one from computing $Y(\Z_S)$ in practice. Our method directly provides the second version of effectiveness without a priori upper bound for height of $t \in Y(\Z_S)$, and we shall describe the procedure as it is implemented in a computer algebra package in order to generate tables of numerical examples.
\par
Recall that the principal tool for us is the map
\begin{align}
\kappa \colon t \mapsto X_t
\end{align}
which associates an elliptic curve $X_t$ to a putative solution $t$. Computation of $\kappa^{-1}(E)$ amounts to solving polynomials in one variable, such as $J_6(x,y,h,E)$. In particular, $\kappa^{-1}(E)$ can be effectively decided once the coefficients of $E$ are known. Thus, what remains is to determine all possible elliptic curves $E$ for which $\kappa^{-1}(E)$ is possibly non-empty.
\begin{theorem}
For a finite set $S$ of prime numbers, the coefficients of isomorphism classes of elliptic curves which have good reduction outside of $S$ can be determined algorithmically.
\end{theorem}
\begin{proof}
We give a brief description. For details, especially practical issues, we refer to Cremona. Using modularity theorem, one may compute the isogeny classes of elliptic curves by means of modular forms. Space of modular forms can be computed using modular symbols for example. For each isogeny class of elliptic curves, one can decide the isomorphism classes of elliptic curves contained in it. In fact, there are at most eight isomorphism classes in a fixed isogeny class, by a theorem of Kenku.
\end{proof}
This modular approach allows us, as a by-product, obtain a new bound for the cardinality of $Y(\Z_S)$.
\begin{theorem}
Let $S$ contain all prime divisors of $2\delta$. Let
\begin{align}
M_1 = \prod_{p\in S} p^{2+d_p}
\end{align}
where $d_2=6$, $d_3=3$, and $d_p=0$ for $p\ge 5$. Let $X_0(M_1)$ be the modular curve of level $\Gamma_0(M_1)$, and let $g_0(M_1)$ be its genus. Then,
\begin{align}
 \left| Y(\Z_S) \right | < 8 \times \left| \mathrm{Aut}_\Q(Y) \right| \times g_0(M_1).
\end{align}
Note that $g_0(M_1) < M_1$.
\end{theorem}
\begin{proof}
The dimension of space of cusp forms of weight two on $X_0(M_1)$ is equal to the genus of $X_0(M_1)$. For each rational Hecke eigenform of weight two, there is at most eight isomorphism classes of elliptic curves by Kenku's theorem. For each elliptic curve, there is at most $\left| \mathrm{Aut}_\Q(Y) \right|$ elements of $Y(\Z_S)$. Thus we obtain the claimed bound.
\end{proof}
\begin{remark}
Above bound is clearly worse than previous ones, as $g_0(M_1)$ is roughly $M_1$. Nevertheless, the proof of above bound uses modularity theorem as its key ingredient, and it has little to do with estimation of number of points on the curve $y^2 = x^3 + n$. 
\end{remark}
In order to compute $Y(\Z_S)$ in practice, one has to first tabulate elliptic curves. By a tabulation of elliptic curves, we shall mean the table of isomorphism classes of elliptic curves, represented in a Weierstrass equation, ordered by their conductors. In a sense this step of tabulation is a pre-computation, which only depends on the discriminant of $h(x,y)$, and the table can be used again and again.

\section{Numerical examples}\label{section:6}
We give numerical examples in this section. As explained before, the crucial step in working out a numerical example is to tabulate elliptic curves with specified conductor. We avoid this step by relying on Cremona's Elliptic Curve Database. In particular, we assume throughout:
\par
\textbf{Hypothesis.} Cremona's Elliptic Curve Datebase is complete(i,e., no curve is omitted) up to conductor $350,000$. 
\par
We will compute $Y(\Z_S)$ based on Cremona's database. The completeness of the list of the solutions depends on the truth of the hypothesis. More specifically, we need a complete list of elliptic curves which has good reduction outside of $2\delta$ and $S$. 
\par
We give some justification for introducing the above hypothesis. If some elliptic curves with conductor less than $350,000$ turn out to be omitted in the Cremona's table, then it is possible that these new elliptic curves give rise to new solutions, which is not listed in the present paper. Such corrections can be always made upon each discovery, if any, of omitted elliptic curves. On the other hand, it is clearly beyond the scope of current work to check that the computer codes which were used to generate Cremona's table contain no bugs.

\subsection{Implementation and perfomance}
In this subsection, we explain how we implement the algorithm into a computer algebra package, and its performance. What we do \emph{not} compute is the necessary table of elliptic curves. We assume that a list of elliptic curves of specified conductor is already available. 
\par
In order to faithfully follow the proof of finiteness of $Y(\Z_S)$, we should compute a set of elliptic curves which contains $\kappa(Y(\Z_S))$, and compute $\kappa^{-1}(E)$ for each curve $E$ in the set. For example, one might choose to compute the set of elliptic curves whose conductor divides the worst conductor of $X_t$. However, this is computationally inefficients because of the following reason. When we replace $E$ by its quadratic twist, $J_6(x,y,h,E)$ is replaced by its multiple. It follows that one has to solve the same polynomial $2^{s+1}$ times when $s$ is the cardinality of $S$. It turns out that working with $X_t'$ we avoid this problem of repeating $J_6(x,y,h,E)$.
\par
The following proposition tells us why it is possible to compute $Y(\Z_S)$ using $X_t'$ instead of $X_t$.
\begin{proposition}\label{prop:6.1}
Let $E$ be an elliptic curve. There exists $t \in Y(\Z_S)$ such that $X_t$ is isomorphic to a quadratic twist of $E$ if and only if $J_6(x,y,h,E)$ has a rational solution.
\end{proposition}
\begin{proof}
Recall that we proved that $J_6(x,y,h,E)$, \eqref{eq:3.32}, and \eqref{eq:3.33} have a common solution if and only if $E$ is isomorphic to $X_t$ itself, without a quadratic twist. Hence, only if part is obvious, and we are left to prove if part of the proposition. It suffices to show that given a solution of $J_6(x,y,h,E)$, one can find a quadratic twist $E'$ of $E$ such that $J_6(x,y,h,E')$, \eqref{eq:3.32}, and \eqref{eq:3.33} have a common solution. 
\par
In order to see that finding such a quadratic twist $E'$ is possible, recall that $J_{24}(x,y,h,E)$ was defined by
\begin{align}
J_{24}(x,y,h,E) := \lambda^{12}\left( c_6(E)^2c_4(x,y,h)^3 - c_4(E)^3 c_6(x,y,h)^2 \right).
\end{align}
Let $E$ is given by the equation
\begin{align}
E \colon y^2 = x^3 + a_4 x + a_6
\end{align}
then let $E'$ be the quadratic twist
\begin{align}
E' \colon y^2 = x^3 + a_4 r^2 x + a_6 r^3
\end{align}
of $E$ by $r$. Then from the formula for $J_{24}(x,y,h,E)$, together with the relations \eqref{eq:3.13} and \eqref{eq:3.14}, tells us that
\begin{align}
J_{24}(x,y,h,E') = r^6 \cdot J_{24}(x,y,h,E).
\end{align}
In particular, from factorisation \eqref{eq:3.35}, it follows that
\begin{align}
J_6(x,y,h,E')=r^6 \cdot J_6(x,y,h,E').
\end{align}
So what remains is to find a value of $r$, for a given zero $(x_0,y_0)$ of $J_6(x,y,h,E')$ (or $J_6(x,y,h,E)$), for which
\begin{align}
\label{eq:6.6}
c_4(x_0,y_0,h) = \lambda^4c_4(E')
\\
\label{eq:6.7}
c_6(x_0,y_0,h) = \lambda^6c_6(E')
\end{align}
have a rational solution in $\lambda$. Above equations are nothing but the equations \eqref{eq:3.32} and \eqref{eq:3.33} written for $E'$. Solving for $\lambda$, we get
\begin{align}
\lambda^2		\,= \,\, &   \frac {c_6(x_0,y_0,h)}{c_4(x_0,y_0,h)} \times \frac{c_4(E')}{c_6(E')}
\\
			\,= \,\,&   \frac  {c_6(x_0,y_0,h)}{c_4(x_0,y_0,h)} \times \frac{c_4(E)}{c_6(E)} \times r^{-1}
\end{align}
so one can take 
\begin{align}
r = \frac  {c_6(x_0,y_0,h)}{c_4(x_0,y_0,h)} \times \frac{c_4(E)}{c_6(E)}
\end{align}
in order to render \eqref{eq:6.6} and \eqref{eq:6.7} to have solution $\lambda = \pm 1$. The assertion of the proposition is proved.
\end{proof}
Based on Proposition~\ref{prop:6.1}, we proceed as follows in order to compute $Y(\Z_S)$. Suppose we are given $h(x,y)$ and $S$. From this, one can compute the list of conductors of $X_t'$ for $t \in Y(\Z_S)$, with negligible computational cost. Using the Cremona's Elliptic Curve Database, we get a sequence of elliptic curves $E_1,E_2, \cdots , E_k$, for some finite $k$, where each $E_i$ is given in a Weierstrass form. Now we compute rational solutions of $J_6(x,y,h,E_i)$ for each $E_i$. If there is a solution, say $x_t$ and $y_t$, then we proceed to compute $h(x_t,y_t)$ and double check that the result is correct.

\par
As we observed before, the coefficients of $J_6(x,y,h,E)$ has a large common factor when we start from an elliptic curve with $a_1=a_2=a_3=0$, we take
\begin{align}
J_6'(x,y,h,E) = 2^{-22} \cdot 3^{-3}\cdot J_6(x,y,h,E)
\end{align}
whose coefficients are still integral.
\par
We give three examples of $h$, and present corresponding $J'_6(x,y,h,E)$, for a generic elliptic curve
\begin{align}
E\colon y^2 = x^3 + a_4x + a_6.
\end{align}
Take
\begin{align}
h_1(x,y) = x^2y - xy^2 = x(x-y)y
\end{align}
which corresponds to the unit equation. In this case,
\begin{align}
J'_6(x,y,h_1,E)=
(x - 2 y)^2   (x + y)^2   (2 x - y)^2a_4^3
+27   (x^2 - x y + y^2)^3a_6^2
\end{align}
which has a particularly nice factorisation. In this case, the solutions $(2,1)$, $(1,-1)$, and $(1,2)$, for $h_1(x,y) = \pm 2^a$ is clearly visible from the factors of coefficients of $a_4$, corresponding to curves with $a_6=0$. As a less trivial example, we take 
\begin{align}
E_\mathrm{960e6}\colon x^3 - x^2 - 21345x - 1190943
\end{align}
which is the minimal Weierstrass equation for elliptic curve \textrm{"960e6"} in Cremona's database. We make change of variables $x\mapsto x+1/3$
\begin{align}
E\colon y^2 = x^3 - \frac{64036}{3}x - \frac{32347568}{27}
\end{align}
for which 
\begin{align}
&J'_6(x,y,h_1,E) 
\\
=& 
-64   (3 x - 128 y)   (3 x + 125 y)   (125 x - 128 y)   (125 x + 3 y)   (128 x - 125 y)   (128 x - 3 y)
\end{align}
and indeed $(3,128)$ belongs to $Y(\Z_S)$ for $S=\{2,3,5\}$, because $128=2^7$, and $128-3=125=5^3$. The other five factors corresponds to orbits of $\mathrm{Aut}_\Q(Y)$. 
\par
As an example for which $\kappa^{-1}(E)$ has no rational point, we take
\begin{align}
E_\mathrm{960e5} \colon  y^2 = x^3 - x^2 - 18465x + 971937
\end{align}
which is the curve \textrm{"960e5"}, which is just next to the previous one. After routine change of variables $x \mapsto x + 1/3$, we obtains
\begin{align}
&J'_6(x,y,h_1,E) 
\\
=& 
-64    (25 x^2 - 13874 x y + 25 y^2)   (25 x^2 + 13824 x y - 13824 y^2)   (13824 x^2 - 13824 x y - 25 y^2)
\end{align}
which is product of three irreducible quadratic polynomials. It follows that there is no solution $t \in Y(\Z_S)$ with $S=\{2,3,5\}$ for which $X_t'$ is isomorphic to $E_\mathrm{960e5}$.
\par
Let us consider a different cubic form
\begin{align}
h_2(x,y) = x^2y+7y^3 = (x^2+7y^2)y
\end{align}
which corresponds to Ramanujan-Nagell equation. In this case,
\begin{align}
J'_6(x,y,h_2,E)=
2^27^4  y^2  (9x^2 + 7y^2)^2a_4^3
-3^37^3  (3x^2 - 7y^2)^3 a_6^2
\end{align}
has again a nice factorisation. Take
\begin{align}
E_\mathrm{210e1}\colon y^2 + xy = x^3 + 210x + 900
\end{align}
which is isomorphic to
\begin{align} 
y^2 = x^3 + \frac{10079}{48} x + \frac{762481}{864}
\end{align}
for which 
\begin{align}
&J'_6(x,y,h_2,E) 
\\
=& 
\frac{7}{1024}   (45 x - 47 y)   (45 x + 47 y)   (2048 x^2 - 14805 x y + 113561 y^2)   (2048 x^2 + 14805 x y + 113561 y^2).
\end{align}
The solution $(47,45)$ corresponds to
\begin{align}
h_2(47,45)	&= 737280 \\
			&= 2^{14} \cdot 3^2 \cdot 5
\end{align}
which is an element of $Y(\Z_S)$ for $S=\{2,3,5\}$.
\par
Finally, take
\begin{align}
h_3(x,y) = x^3 - x^2y - 4xy^2 - y^3
\end{align}
whose discriminant is $13^2$. In this case,
\begin{align}
J'_6(x,y,h_3,E)=
13^2  (5x^3 + 21x^2y + 6xy^2 - 5y^3)^2a_4^3
+3^313^3  (x^2 + xy + y^2)^3a_6^2
\end{align}
which shows that the coefficients of $a_4^3$ in general have irreducible factor of degree three. Of course, the squares and cubes in the coefficients are expected, as we have defined $J_6(x,y,h,E)$ by \eqref{eq:3.34} and \eqref{eq:3.35}.
\par
In practice, the coefficients $a_4$ and $a_6$ are very large, so it is difficult to factor $J'_6(x,y,h,E)$ manually. Nonetheless, using suitable computer algebra packages, one can factor such polynomials rather quickly. The author's experience shows that SageMathCloud is able to factor roughly 200 polynomials per second.
\par
Take $h(x,y)=h_1(x,y)$ as above, and take $S=\{2,7,11,13\}$. Then we need a table of elliptic curves whose conductor divides $2^8 \cdot 7 \cdot 11 \cdot 13 = 256256$. Since $256256<350000$ we may use Cremona's Elliptic Curve Database, from which we get 940 such curves. From them, we get $51$ solutions, as we display in Table~\ref{01-10271113}. It took 4.13 seconds in CPU time to generate Table~\ref{01-10271113}. A typical box in the table looks like
\begin{align}
\begin{array}{|c|}
\hline (x_t,y_t) \\h(x_t,y_t) \\ \text{Factorisation of $h(x_t,y_t)$} \\ \textrm{Cremona's label} \\ \text{Factorisation of the conductor}
\\
\hline
\end{array}
\end{align}
with five items listed vertically. 

\par
Let us take $S=\{2,3,5,7\}$ for $h_2(x,y)$. In this case, we need elliptic curves of conductor dividing $ 2^8 \cdot 3 \cdot 5 \cdot 7^2 = 188160$, which is smaller than $350000$. The Cremona's table contains 4568 such curves, from which we find $33$ solutions. It took 20.09 seconds in CPU time to generate Table~\ref{01072357}.

\par
Let us take $S=\{2,5,13\}$ for $h_3(x,y)$. In this case, we need elliptic curves of conductor dividing $ 2^8 \cdot 5 \cdot 13^2 = 216320$, which is smaller than $350000$. The Cremona's table contains 976 such curves, from which we find $35$ solutions. It took 4.20 seconds in CPU time to generate Table~\ref{1-1-4-12513}. Since we wrote a computer code which compute the solutions with $y_t\not = 0$, the trivial solution $h(1,0)=1$ is omitted from the table .
\begin{remark}
Repeated $6$'s on the exponent of $2$ in the conductors appearing in Table~\ref{01-10271113} can be perhaps predicted by considering connected components in the Neron model. On the other hand, as the $2$ divides the discriminant of $h_2(x,y)$, various exponents of $2$ appear in Table~\ref{01072357}. Although we ignored further analysis of conductors of $X_t'$ at the primes dividing $2\delta$, such an analysis might help practical computations.
\end{remark}

\subsection{Statistics for $x(x-y)y$}
In this section, we fix
\begin{align}
h(x,y) = x(x-y)y
\end{align}
and vary $S$ in a few directions. Let us begin with the case when
\begin{align}
S = \{2,p\}
\end{align}
consists of $2$ and one more prime $p$. Cremona's database allows us to compute $Y(\Z_S)$ if
\begin{align}
{2^8\cdot p}<{350000}
\end{align}
or $p \le 1367$. It follows that except for $p=5,7,17,31,257$, we have
\begin{align}
Y(\Z_S) = Y(\Z_{\{2\}}) = \{(2:1),(1:-1),(1:2)\}.
\end{align}
In fact, this case is less interesting since computation of $Y(\Z_S)$ reduces to 
\begin{align}
2^m - p^n = \pm 1
\end{align}
which is a special case of Catalan's equation. Since Catalan's conjecture is known, solutions of above equation necessarily satisfies $n=1$, and $Y(\Z_S)$ has more than three elements if and only if $p$ is a prime of the form $2^m \pm 1$. Thus, we are merely verifying Catalan's conjecture. 
\par
As another example, take
\begin{align}
S = \{2,3,p\}
\end{align}
consists of 2, 3 and another prime $p>3$. Since
\begin{align}
\frac{350000}{2^8\cdot 3}<456
\end{align}
one can use Cremona's database for primes $p$ up to 449, which is the 87th prime number. In the next table we display
\begin{align*}
\begin{array}{|c|}
\hline
p:m
\\
\hline
\end{array}
\end{align*}
as $p$ varies among the primes from $5$ to $449$, and $Y(\Z_S)$ has $3+6m$ solutions. 
\begin{align*}
\begin{array}{|c|c|c|c|c|c|c|c|c|c|c|c}
\hline
5:16
&7:12
&11:9
&13:8
&17:8
&19:7
&23:6
&29:5
&31:5
\\\hline
37:5
&41:5
&43:5
&47:5
&53:4
&59:4
&61:5
&67:4
&71:4
\\\hline
73:6
&79:4
&83:4
&89:4
&97:5
&101:4
&103:3
&107:4
&109:4
\\\hline
113:4
&127:4
&131:4
&137:4
&139:4
&149:3
&151:3
&157:3
&163:4
\\\hline
167:3
&173:3
&179:4
&181:3
&191:4
&193:4
&197:3
&199:3
&211:4
\\\hline
223:3
&227:4
&229:4
&233:3
&239:4
&241:4
&251:4
&257:4
&263:3
\\\hline
269:4
&271:3
&277:3
&281:3
&283:4
&293:3
&307:4
&311:3
&313:3
\\\hline
317:3
&331:3
&337:4
&347:3
&349:3
&353:3
&359:3
&367:3
&373:3
\\\hline
379:3
&383:4
&389:3
&397:3
&401:3
&409:3
&419:3
&421:3
&431:5
\\\hline
433:4
&439:3
&443:3
&449:3
&//
&//
&//
&//
&//
\\
\hline
\end{array}
\end{align*}
Note that for $S=\{2,3\}$, $Y(\Z_S)$ has $21$ elements, so in particular $m=3$. Above table indicates that $m$ stabilises around $3$, with a notable exception for $p=431$. It is mainly due to a rather surprising identity $431 = 2^9 - 3^4$. 
\par
Now we take
\begin{align}
S = \{2,3,5,p \}
\end{align}
where $p$ is a prime number which does not exceed $89$. Note that $Y_{\{2,3,5 \}}$ contains $99=3 + 6 \times 16$ elements, so a trivial lower bound for $m$ in this case is $16$. Using our algorithm, we obtain the following table.
\begin{align*}
\begin{array}{|c|c|c|c|c|c|c|c|c|c|c|c|c|}
 \hline
 7:62&
 11:46&
 13:44&
 17:37&
 19:37&
 23:35&
 29:31&
 31:30&
 37:30&
 41:30&
 43:28
 \\
 \hline
 47:26&
 53:28&
 59:25&
 61:26&
 67:26&
 71:25&
 73:25&
 79:25&
 83:25&
 89:23&
 //
 \\
 \hline
\end{array}
\end{align*}

It indicates that $Y(\Z_S)$ steadily decreases as $p$ increases, although it is unclear whether it will reach $m=16$ at some point. Note an exceptional increment at $p=53$, for which we record the solutions and associated Cremona label in Table~\ref{01-1023553}. In fact, the number of possible elliptic curves tends to decrease as $p$ increases. For example, if $p=7$, there are 1688 curves, while the corresponding number is 1080 for $p=87$.
\par
One wonders weather one can improve the bound on the cardinality of $Y(\Z_S)$, as $S$ varies among certain subsets of prime numbers with fixed cardinality, such as $S = S_0 \cup \{p \}$ with varying $p$.

\section{Comparison with the work of Tzanakis and de Weger.}\label{section:7}
There had been an attempt to explicitly solve Thue-Mahler equation by Tzanakis and de Weger, based on linear forms in logarithms. Theoretical foundations for two approaches are quite different, and in this section we compare the two from practical point of view. 
\par
In our approach, the computation of $\kappa^{-1}(E)$ for a given $E$ is easy. There is a formula for $J_6(x,y,h,E)$ to which we plug in the coefficients of $E$, and the rational solutions of $J_6(x,y,h,E)$ can be found quickly. On the other hand, finding all possible candidates of $E$ is difficult, although the modularity of elliptic curves significantly facilitates it. The upshot is that in our approach, the modularity of elliptic curves reduces the determination of all possible elliptic curves to a single linear algebra problem on the space of modular forms. One can model the space of modular forms using the space of modular symbols. As Cremona explains in Section~2 of his article \cite{Cremona}, computation of modular symbols of a given level can be done rather quickly. On this $\Q$-vector space of modular symbols, whose dimension is quite large, one has to find all one-dimensional Hecke invariant rational subspaces, and compute sufficiently many Hecke eigenvalues in order to find approximated $c_4$ and $c_6$ invariants of a curve in the corresponding isogeny class. Given an isogeny class, it is not so difficult to determine all isomorphism classes which belong to it. To summarise, this linear algebra problem on a huge space of modular symbols seems to lie at the bottleneck of our process. 
\par
The approach of Tzanakis and de Weger is, as explained in the introduction of \cite{Tzanakis de Weger}, consists of three steps. The first is to obtain large bounds from estimation of linear forms in logarithms of possibly irrational algebraic numbers. The second is to reduce the large bounds using diophantine approximation. The last step is to search for solutions below the bound, not by brutal force, but by using an algorithm to search for lattice points on a given sphere, a sieving process, and enumeration of possibilities. The authors remarks that the third process might well be the computational bottleneck for their process. They worked out two following concrete examples
\begin{align}
x^3 - 23 x^2y + 5xy^2 +24y^3 &= \pm 2^{e_1} 3^{e_2} 5^{e_3} 7^{e_4}, &\delta &= 5^2 \cdot 44621
\\
x^3 - 3xy^2 - y^3 &= \pm 3^{e_1}  17^{e_2} 19^{e_3}, &\delta &= 3^4
\end{align}
using their method. In order to solve above equation using our method, we need elliptic curves with conductor dividing $2^8\cdot 3 \cdot 5^2 \cdot 7 \cdot 44621^2>5 \times 10^{13}$, and $2^8 \cdot 3^5 \cdot 17 \cdot 19 >2 \times 10^7$, which are not provided by Cremona's database.
\par
We observe that the tools of our method has little to do with those of Tzanakis and de Weger. The computational bottlenecks of two approaches are different: ours is in linear algebra while theirs (seems to) be in geometry of numbers. As we vary $h(x,y)$, we observe another difference. The computations we need to carry are mostly insensitive to $h(x,y)$ except for the discriminant and $S$. Once the database of elliptic curves is established one can use the same data for a different $h(x,y)$. The computations of Tzanakis and de Weger depends on the specific $S$-unit equation which is sensitive to a chosen zero of $h(x,y)$.

\section{Generalised Ramanujan-Nagell equation}\label{section:8}
The goal of current section is twofold. Firstly, we shall consider a special form of $h(x,y)$ and determine $Y(\Z_S)$, from which we deduce the complete set of solutions of certain generalised Ramanujan-Nagell equations. Secondly, we shall analyse the statistical behaviour as we vary $S$ as $h(x,y)$ remains fixed.
\par
The equation
\begin{align}\label{eq:8.1}
x^2 + 7 = 2^n
\end{align}
for integers $x$ and $n$, is often called the Ramanujan-Nagell equation in the literature. One can relate it to Thue-Mahler equation, since if we take
\begin{align}\label{eq:8.2}
h(x,y) = (x^2 + dy^2)y
\end{align}
then the solution $(x,n)=(x_0,n_0)$ of the Ramanujan-Nagell equation leads to the point $(x,y)=(x_0,1) \in Y(\Z_S)$ for $d=7$ and $S=\{2\}$. Conversely we can recover the solutions of the Ramanujan-Nagell equation by computing $Y(\Z_S)$. Thus one may consider Thue-Mahler equations for \eqref{eq:8.2} as a generalisation of Ramanujan-Nagell equation.
\par
In Table~\ref{01072357}, we display the elements of $Y(\Z_S)$ for $d=7$ and $S=\{2,3,5,7\}$, from which we conclude
\begin{align}
x^2 + 7 = 2^{e_1}3^{e_2}5^{e_3}7^{e_4}
\end{align}
with positive $x$ has seven solutions corresponding to $x=181,21,11,7,5,3$ and $1$. In Table~\ref{0107271113}, we find that
\begin{align}
x^2 + 7 = 2^{e_1}7^{e_2}11^{e_3}13^{e_4}
\end{align}
in positive $x$ has fourteen solutions corresponding to $x=273$, $181$, $75$, $53$, $35$, $31$, $21$, $13$, $11$, $9$, $7$, $5$, $3$, and $1$.
\par
For $d=-7$, and $S=\{2,5,7,11\}$, Table~\ref{010-725711} shows that  we have particularly few solutions. Indeed $Y(\Z_S)$ has five elements, among which three elements satisfy $y=1$. In particular, 
\begin{align}
x^2 - 7 = 2^{e_1}5^{e_2}7^{e_3}11^{e_4}
\end{align}
has only one solution $x=3$ among positive integers.
\par
As we vary $S$ for a fixed $d$, such that $-d$ is not a square, we observe a pattern which we describe now. The pattern seems to persist for any such $d$, but let us take $d=1$ for clarity. In particular, we consider
\begin{align}
h(x,y) = (x^2 + y^2)y
\end{align}
in the rest of the present section. 
\par
Take $S=\{2,p\}$, for a prime $p \ge 3$. Let the cardinality of $Y(\Z_{\{2,p\}})$ to be $m$. We shall divide it into two cases, depending on whether or not $-1$ is a quadratic residue modulo $p$, and compare the variation of $m$. Note that $Y(\Z_{\{2\}})$ has three elements corresponding to $x=0,1$ and $-1$. The table below lists $$\begin{array}{|c|}\hline p:m \\\hline \end{array}$$ in the range of $p$ for which $-1$ is a quadratic residue, and $p \le 113$. 
\begin{align}
\begin{array}{|c|c|c|c|c|c|c|c|c|c|c|c|c|c|c|}
\hline
5:15	&
13:9	&
17:9	&
29:7	&
37:5	&
41:7	&
53:5	
\\
\hline
61:5	&
73:5	&
89:5	&
97:5	&
101:5	&
109:3	&
113:7	\\	
\hline
\end{array}
\end{align}
On the other hand, we observe that
\begin{align}
Y(\Z_{\{2,p\}}) = Y(\Z_{\{2\}})
\end{align}
if $-1$ is quadratic non-residue modulo $p$ and $p \le 113$.
\par
The different behaviour of the cardinality of $Y(\Z_S)$ continues when we take $S=\{2,3,p\}$. We can numerically verify that 
\begin{align}
Y(\Z_{\{2,3,p\}}) = Y(\Z_{\{2\}})
\end{align}
if $-1$ is quadratic non-residue modulo $p$, and $p < 455$. In contrast, if $-1$ is a quadratic residue modulo $p$ and $p < 455$, then $Y(\Z_{\{2,3,p\}})$ is strictly larger than $Y(\Z_{\{2\}})$ all the time. For instance, for $p=449$,
\begin{align}
(13^2 + 27^2)\cdot 27 = 24246 = 3 \cdot 3^3 \cdot 449
\end{align}
is associated to the curve \textrm{"172416o1"}, and similarly for $p = 433$,
\begin{align}
(17^2 + 12^2)\cdot 12 = 5196 = 2^4 \cdot 3 \cdot 433
\end{align}
is associated to the curve \textrm{"20784e2"}. The numerical data suggests that the cardinality of $Y(\Z_{\{2,3,p\}})$ is exactly five for most of $p$.
\par
Now we take $S = \{2,5,p\}$. Note that $Y(\Z_{\{2,5\}})$ consists of fifteen elements. Based on the previous observation, one might conjecture that $Y(\Z_{\{2,5,p\}})$ consists of fifteen elements if $-1$ is quadratic non-residue module $p$. However, there is a counterexample for $p = 139$. Indeed, one finds
\begin{align}
(29^2+278^2)\cdot 278 =  2 \cdot 5^7 \cdot 139
\end{align}
which is associated to the curve \textrm{"11120e2"}. Among the primes $5< p \le 271$ for which $-1$ is a quadratic non-residue, we verify
\begin{align}
Y(\Z_{\{2,5,p\}}) = Y(\Z_{\{2,5\}})
\end{align}
holds except $p=7,11,19,31,79, 139,191$. In these exceptional cases, $Y(\Z_{\{2,5,p\}}) - Y(\Z_{\{2,5\}})$ contains four elements when $p=7,11$, and two elements in the remaining five cases. In contrast, if we take $S=\{2,5,p\}$ for a prime $p$ for which $-1$ is a quadratic residue, then $Y(\Z_{\{2,5,p\}})$ is strictly larger than $Y(\Z_{\{2,5\}})$, in the range $5 <p \le 271$. The smallest cardinality of $Y(\Z_{\{2,5,p\}})$ is seventeen, which happens precisely for $p=241$. Up to sign of $x_t$ there is a unique element of $Y(\Z_{\{2,5,241\}})$ which does not belong to $Y(\Z_{\{2,5\}})$, which is
\begin{align}
(15^2+4^2)4 =  964 = 2^2 \cdot 241 
\end{align}
associated to the curve \textrm{"15424d2"}.
\par
We conclude by formulating a precise question based on our observation. Let $S_0$ be a fixed set of primes. For a positive real number $X$ and $i \in \{1,3\}$, let $\pi_i(X)$ be the number of primes $p$ less than $X$ which are congruent to $i$ modulo $4$. Define
\begin{align}
A_1(S_0) = \liminf_{X \to \infty} \frac{1}{\pi_1(X)}\left(\sum_{p<X, p\equiv 1 \textrm{(mod $4$)}} \left | Y(\Z_{S_0\cup \{p\}})\right|\right)
\end{align}
and
\begin{align}
A_3 (S_0)= \limsup_{X \to \infty} \frac{1}{\pi_3(X)}\left(\sum_{p<X, p\equiv 3 \textrm{(mod $4$)}} \left | Y(\Z_{S_0\cup \{p\}})\right|\right).
\end{align}
One might speculate that $A_1(\{2\}) = 5$ and $A_3(\{2\}))=3$. To be on a conservative side, one might ask whether
\begin{align}
A_1(S_0) > A_3(S_0)
\end{align}
holds. The author is not able to show 
\begin{align}
A_1(S_0) \ge A_3(S_0)
\end{align}
holds for any particular $S_0$.
\clearpage

\begin{table}
\caption{$(a,b,c,d)=(0,1,-1,0)$, $\delta=1$, $ S = \{{2, 7,11, 13}\}$}\label{01-10271113}
\begin{center}
\footnotesize
\begin{tabular}{|c|c|c|c|c|c|c|c|c|}
\hline
$
\begin{array}{c}
(-343,169) \\ 29679104 \\ 2^9 \cdot 7^3 \cdot 13^2 \\ \textrm{5824bd2} \\ 2^6 \cdot 7 \cdot 13
\end{array}
$
&
$
\begin{array}{c}
(-169,7) \\ 208208 \\ 2^4 \cdot 7 \cdot 11 \cdot 13^2 \\ \textrm{64064a2} \\ 2^6 \cdot 7 \cdot 11 \cdot 13
\end{array}
$
&
$
\begin{array}{c}
(-169,343) \\ 29679104 \\ 2^9 \cdot 7^3 \cdot 13^2 \\ \textrm{5824bd2} \\ 2^6 \cdot 7 \cdot 13
\end{array}
$
&
$
\begin{array}{c}
(-121,7) \\ 108416 \\ 2^7 \cdot 7 \cdot 11^2 \\ \textrm{4928h2} \\ 2^6 \cdot 7 \cdot 11
\end{array}
$
&
$
\begin{array}{c}
(-64,13) \\ 64064 \\ 2^6 \cdot 7 \cdot 11 \cdot 13 \\ \textrm{64064k2} \\ 2^6 \cdot 7 \cdot 11 \cdot 13
\end{array}
$
\\\hline
$
\begin{array}{c}
(-13,1) \\ 182 \\ 2 \cdot 7 \cdot 13 \\ \textrm{5824b2} \\ 2^6 \cdot 7 \cdot 13
\end{array}
$
&
$
\begin{array}{c}
(-13,64) \\ 64064 \\ 2^6 \cdot 7 \cdot 11 \cdot 13 \\ \textrm{64064k2} \\ 2^6 \cdot 7 \cdot 11 \cdot 13
\end{array}
$
&
$
\begin{array}{c}
(-11,2) \\ 286 \\ 2 \cdot 11 \cdot 13 \\ \textrm{9152bd2} \\ 2^6 \cdot 11 \cdot 13
\end{array}
$
&
$
\begin{array}{c}
(-7,1) \\ 56 \\ 2^3 \cdot 7 \\ \textrm{448b2} \\ 2^6 \cdot 7
\end{array}
$
&
$
\begin{array}{c}
(-7,4) \\ 308 \\ 2^2 \cdot 7 \cdot 11 \\ \textrm{4928m2} \\ 2^6 \cdot 7 \cdot 11
\end{array}
$
\\\hline
$
\begin{array}{c}
(-7,121) \\ 108416 \\ 2^7 \cdot 7 \cdot 11^2 \\ \textrm{4928h2} \\ 2^6 \cdot 7 \cdot 11
\end{array}
$
&
$
\begin{array}{c}
(-7,169) \\ 208208 \\ 2^4 \cdot 7 \cdot 11 \cdot 13^2 \\ \textrm{64064a2} \\ 2^6 \cdot 7 \cdot 11 \cdot 13
\end{array}
$
&
$
\begin{array}{c}
(-4,7) \\ 308 \\ 2^2 \cdot 7 \cdot 11 \\ \textrm{4928m2} \\ 2^6 \cdot 7 \cdot 11
\end{array}
$
&
$
\begin{array}{c}
(-2,11) \\ 286 \\ 2 \cdot 11 \cdot 13 \\ \textrm{9152bd2} \\ 2^6 \cdot 11 \cdot 13
\end{array}
$
&
$
\begin{array}{c}
(-1,1) \\ 2 \\ 2 \\ \textrm{256c2} \\ 2^8
\end{array}
$
\\
\hline
\end{tabular}
\end{center}
\label{default}

\end{table}


\begin{table}
\caption{$(a,b,c,d)=(0,1,-1,0)$, $\delta=1$, $ S = \{{2, 3, 431}\}$}\label{01-1023431}
\begin{center}
\footnotesize
\begin{tabular}{|c|c|c|c|c|c|c|c|c|}
\hline
$
\begin{array}{c}
(-431,1) \\ 186192 \\ 2^4\cdot 3^3\cdot 431 \\ \textrm{82752bd2} \\ 2^6\cdot 3\cdot 431
\end{array}
$
&
$
\begin{array}{c}
(-431,81) \\ 17874432 \\ 2^9\cdot 3^4\cdot 431 \\ \textrm{82752t2} \\ 2^6\cdot 3\cdot 431
\end{array}
$
&
$
\begin{array}{c}
(-81,431) \\ 17874432 \\ 2^9\cdot 3^4\cdot 431 \\ \textrm{82752t2} \\ 2^6\cdot 3\cdot 431
\end{array}
$
&
$
\begin{array}{c}
(-8,1) \\ 72 \\ 2^3\cdot 3^2 \\ \textrm{192c3} \\ 2^6\cdot 3
\end{array}
$
&
$
\begin{array}{c}
(-3,1) \\ 12 \\ 2^2\cdot 3 \\ \textrm{192d2} \\ 2^6\cdot 3
\end{array}
$
&
$
\begin{array}{c}
(-2,1) \\ 6 \\ 2\cdot 3 \\ \textrm{192b2} \\ 2^6\cdot 3
\end{array}
$
\\\hline
$
\begin{array}{c}
(-1,1) \\ 2 \\ 2 \\ \textrm{256c2} \\ 2^8
\end{array}
$
&
$
\begin{array}{c}
(-1,2) \\ 6 \\ 2\cdot 3 \\ \textrm{192b2} \\ 2^6\cdot 3
\end{array}
$
&
$
\begin{array}{c}
(-1,3) \\ 12 \\ 2^2\cdot 3 \\ \textrm{192d2} \\ 2^6\cdot 3
\end{array}
$
&
$
\begin{array}{c}
(-1,8) \\ 72 \\ 2^3\cdot 3^2 \\ \textrm{192c3} \\ 2^6\cdot 3
\end{array}
$
&
$
\begin{array}{c}
(-1,431) \\ 186192 \\ 2^4\cdot 3^3\cdot 431 \\ \textrm{82752bd2} \\ 2^6\cdot 3\cdot 431
\end{array}
$
&
$
\begin{array}{c}
(1,2) \\ -2 \\ -1\cdot 2 \\ \textrm{256c2} \\ 2^8
\end{array}
$
\\\hline
$
\begin{array}{c}
(1,3) \\ -6 \\ -1\cdot 2\cdot 3 \\ \textrm{192b2} \\ 2^6\cdot 3
\end{array}
$
&
$
\begin{array}{c}
(1,4) \\ -12 \\ -1\cdot 2^2\cdot 3 \\ \textrm{192d2} \\ 2^6\cdot 3
\end{array}
$
&
$
\begin{array}{c}
(1,9) \\ -72 \\ -1\cdot 2^3\cdot 3^2 \\ \textrm{192c3} \\ 2^6\cdot 3
\end{array}
$
&
$
\begin{array}{c}
(1,432) \\ -186192 \\ -1\cdot 2^4\cdot 3^3\cdot 431 \\ \textrm{82752bd2} \\ 2^6\cdot 3\cdot 431
\end{array}
$
&
$
\begin{array}{c}
(2,1) \\ 2 \\ 2 \\ \textrm{256c2} \\ 2^8
\end{array}
$
&
$
\begin{array}{c}
(2,3) \\ -6 \\ -1\cdot 2\cdot 3 \\ \textrm{192b2} \\ 2^6\cdot 3
\end{array}
$
\\
\hline
\end{tabular}
\end{center}
\label{default}
\end{table}

\begin{table}
\caption{$(a,b,c,d)=(0,1,-1,0)$, $\delta=1$, $ S = \{{2, 3, 5, 53}\}$}\label{01-1023553}
\begin{center}
\footnotesize
\begin{tabular}{|c|c|c|c|c|c|c|c|c|}
\hline
$
\begin{array}{c}
(-6561,64) \\ \textrm{50880cd2}
\end{array}
$
&
$
\begin{array}{c}
(-3072,53) \\ \textrm{50880bq2}
\end{array}
$
&
$
\begin{array}{c}
(-256,9) \\ \textrm{50880y2}
\end{array}
$
&
$
\begin{array}{c}
(-159,1) \\ \textrm{50880bx2}
\end{array}
$
&
$
\begin{array}{c}
(-125,3) \\ \textrm{960e6}
\end{array}
$
&
$
\begin{array}{c}
(-81,25) \\ \textrm{50880bw2}
\end{array}
$
\\\hline
$
\begin{array}{c}
(-80,1) \\ \textrm{960g5}
\end{array}
$
&
$
\begin{array}{c}
(-75,53) \\ \textrm{50880dz2}
\end{array}
$
&
$
\begin{array}{c}
(-72,53) \\ \textrm{50880bv2}
\end{array}
$
&
$
\begin{array}{c}
(-64,6561) \\ \textrm{50880cd2}
\end{array}
$
&
$
\begin{array}{c}
(-53,1) \\ \textrm{10176f2}
\end{array}
$
&
$
\begin{array}{c}
(-53,27) \\ \textrm{50880s2}
\end{array}
$
\\\hline
$
\begin{array}{c}
(-53,72) \\ \textrm{50880bv2}
\end{array}
$
&
$
\begin{array}{c}
(-53,75) \\ \textrm{50880dz2}
\end{array}
$
&
$
\begin{array}{c}
(-53,3072) \\ \textrm{50880bq2}
\end{array}
$
&
$
\begin{array}{c}
(-50,3) \\ \textrm{50880cw2}
\end{array}
$
&
$
\begin{array}{c}
(-48,5) \\ \textrm{50880bt2}
\end{array}
$
&
$
\begin{array}{c}
(-45,8) \\ \textrm{50880ea2}
\end{array}
$
\\\hline
$
\begin{array}{c}
(-27,5) \\ \textrm{960o2}
\end{array}
$
&
$
\begin{array}{c}
(-27,53) \\ \textrm{50880s2}
\end{array}
$
&
$
\begin{array}{c}
(-25,2) \\ \textrm{960d2}
\end{array}
$
&
$
\begin{array}{c}
(-25,81) \\ \textrm{50880bw2}
\end{array}
$
&
$
\begin{array}{c}
(-24,1) \\ \textrm{960b3}
\end{array}
$
&
$
\begin{array}{c}
(-16,9) \\ \textrm{960g3}
\end{array}
$
\\\hline
$
\begin{array}{c}
(-15,1) \\ \textrm{960i2}
\end{array}
$
&
$
\begin{array}{c}
(-9,1) \\ \textrm{960j2}
\end{array}
$
&
$
\begin{array}{c}
(-9,16) \\ \textrm{960g3}
\end{array}
$
&
$
\begin{array}{c}
(-9,256) \\ \textrm{50880y2}
\end{array}
$
&
$
\begin{array}{c}
(-8,1) \\ \textrm{192c3}
\end{array}
$
&
$
\begin{array}{c}
(-8,45) \\ \textrm{50880ea2}
\end{array}
$
\\
\hline
\end{tabular}
\end{center}
\label{default}
\end{table}


\begin{table}
\caption{$(a,b,c,d)=(0,1,0,7)$, $\delta=-1\cdot 2^2\cdot 7$, $ S = \{{2, 3, 5, 7}\}$}\label{01072357}
\begin{center}
\begin{tabular}{|c|c|c|c|c|c|c|c|c|}
\hline
$
\begin{array}{c}
(-181,1) \\ 32768 \\ 2^{15} \\ \textrm{3136y5} \\ 2^6\cdot 7^2
\end{array}
$
&
$
\begin{array}{c}
(-119,5) \\ 71680 \\ 2^{11}\cdot 5\cdot 7 \\ \textrm{15680v1} \\ 2^6\cdot 5\cdot 7^2
\end{array}
$
&
$
\begin{array}{c}
(-47,45) \\ 737280 \\ 2^{14}\cdot 3^2\cdot 5 \\ \textrm{210e1} \\ 2\cdot 3\cdot 5\cdot 7
\end{array}
$
&
$
\begin{array}{c}
(-35,9) \\ 16128 \\ 2^8\cdot 3^2\cdot 7 \\ \textrm{294g1} \\ 2\cdot 3\cdot 7^2
\end{array}
$
&
$
\begin{array}{c}
(-31,3) \\ 3072 \\ 2^{10}\cdot 3 \\ \textrm{9408p1} \\ 2^6\cdot 3\cdot 7^2
\end{array}
$
\\\hline
$
\begin{array}{c}
(-21,1) \\ 448 \\ 2^6\cdot 7 \\ \textrm{3136r3} \\ 2^6\cdot 7^2
\end{array}
$
&
$
\begin{array}{c}
(-13,7) \\ 3584 \\ 2^9\cdot 7 \\ \textrm{14a1} \\ 2\cdot 7
\end{array}
$
&
$
\begin{array}{c}
(-11,1) \\ 128 \\ 2^7 \\ \textrm{3136k1} \\ 2^6\cdot 7^2
\end{array}
$
&
$
\begin{array}{c}
(-9,5) \\ 1280 \\ 2^8\cdot 5 \\ \textrm{70a1} \\ 2\cdot 5\cdot 7
\end{array}
$
&
$
\begin{array}{c}
(-7,1) \\ 56 \\ 2^3\cdot 7 \\ \textrm{3136h1} \\ 2^6\cdot 7^2
\end{array}
$
\\\hline
$
\begin{array}{c}
(-7,3) \\ 336 \\ 2^4\cdot 3\cdot 7 \\ \textrm{2352g1} \\ 2^4\cdot 3\cdot 7^2
\end{array}
$
&
$
\begin{array}{c}
(-7,5) \\ 1120 \\ 2^5\cdot 5\cdot 7 \\ \textrm{3920o1} \\ 2^4\cdot 5\cdot 7^2
\end{array}
$
&
$
\begin{array}{c}
(-5,1) \\ 32 \\ 2^5 \\ \textrm{3136bb1} \\ 2^6\cdot 7^2
\end{array}
$
&
$
\begin{array}{c}
(-3,1) \\ 16 \\ 2^4 \\ \textrm{112b1} \\ 2^4\cdot 7
\end{array}
$
&
$
\begin{array}{c}
(-1,1) \\ 8 \\ 2^3 \\ \textrm{224a1} \\ 2^5\cdot 7
\end{array}
$
\\
\hline
\end{tabular}
\end{center}
\label{default}
\end{table}

\begin{table}
\caption{$(a,b,c,d)=(0,1,0,7)$, $\delta=-1\cdot 2^2\cdot 7$, $ S = \{{2,  7,11,13}\}$}\label{0107271113}
\begin{center}
\footnotesize
\begin{tabular}{|c|c|c|c|c|c|c|c|c|}
\hline
$
\begin{array}{c}
(-273,1) \\ 74536 \\ 2^3\cdot 7\cdot 11^3 \\ \textrm{34496dc1} \\ 2^6\cdot 7^2\cdot 11
\end{array}
$
&
$
\begin{array}{c}
(-181,1) \\ 32768 \\ 2^{15} \\ \textrm{3136y5} \\ 2^6\cdot 7^2
\end{array}
$
&
$
\begin{array}{c}
(-161,13) \\ 352352 \\ 2^5\cdot 7\cdot 11^2\cdot 13 \\ \textrm{112112o1} \\ 2^4\cdot 7^2\cdot 11\cdot 13
\end{array}
$
&
$
\begin{array}{c}
(-87,91) \\ 5963776 \\ 2^{16}\cdot 7\cdot 13 \\ \textrm{182a1} \\ 2\cdot 7\cdot 13
\end{array}
$
&
$
\begin{array}{c}
(-81,13) \\ 100672 \\ 2^6\cdot 11^2\cdot 13 \\ \textrm{112112bf1} \\ 2^4\cdot 7^2\cdot 11\cdot 13
\end{array}
$
&
$
\begin{array}{c}
(-75,1) \\ 5632 \\ 2^9\cdot 11 \\ \textrm{34496bf1} \\ 2^6\cdot 7^2\cdot 11
\end{array}
$
\\\hline
$
\begin{array}{c}
(-67,49) \\ 1043504 \\ 2^4\cdot 7^2\cdot 11^3 \\ \textrm{1232d1} \\ 2^4\cdot 7\cdot 11
\end{array}
$
&
$
\begin{array}{c}
(-57,11) \\ 45056 \\ 2^{12}\cdot 11 \\ \textrm{34496cf1} \\ 2^6\cdot 7^2\cdot 11
\end{array}
$
&
$
\begin{array}{c}
(-53,1) \\ 2816 \\ 2^8\cdot 11 \\ \textrm{34496do1} \\ 2^6\cdot 7^2\cdot 11
\end{array}
$
&
$
\begin{array}{c}
(-49,13) \\ 46592 \\ 2^9\cdot 7\cdot 13 \\ \textrm{1274j1} \\ 2\cdot 7^2\cdot 13
\end{array}
$
&
$
\begin{array}{c}
(-35,1) \\ 1232 \\ 2^4\cdot 7\cdot 11 \\ \textrm{34496z1} \\ 2^6\cdot 7^2\cdot 11
\end{array}
$
&
$
\begin{array}{c}
(-31,1) \\ 968 \\ 2^3\cdot 11^2 \\ \textrm{34496dj1} \\ 2^6\cdot 7^2\cdot 11
\end{array}
$
\\
\hline
\end{tabular}
\end{center}
\label{default}
\end{table}


\begin{table}
\caption{$(a,b,c,d)=(0,1,0,2)$, $\delta=-1\cdot 2^3$, $ S = \{{ 2,3, 11}\}$}\label{01022311}
\begin{center}
\footnotesize
\begin{tabular}{|c|c|c|c|c|c|c|c|c|}
\hline
$
\begin{array}{c}
(-695,8) \\ 3865224 \\ 2^3\cdot 3\cdot 11^5 \\ \textrm{2112d4} \\ 2^6\cdot 3\cdot 11
\end{array}
$
&
$
\begin{array}{c}
(-241,22) \\ 1299078 \\ 2\cdot 3^{10}\cdot 11 \\ \textrm{2112n4} \\ 2^6\cdot 3\cdot 11
\end{array}
$
&
$
\begin{array}{c}
(-155,4) \\ 96228 \\ 2^2\cdot 3^7\cdot 11 \\ \textrm{2112j2} \\ 2^6\cdot 3\cdot 11
\end{array}
$
&
$
\begin{array}{c}
(-140,1) \\ 19602 \\ 2\cdot 3^4\cdot 11^2 \\ \textrm{8448e2} \\ 2^8\cdot 3\cdot 11
\end{array}
$
&
$
\begin{array}{c}
(-59,16) \\ 63888 \\ 2^4\cdot 3\cdot 11^3 \\ \textrm{2112e4} \\ 2^6\cdot 3\cdot 11
\end{array}
$
&
$
\begin{array}{c}
(-50,9) \\ 23958 \\ 2\cdot 3^2\cdot 11^3 \\ \textrm{8448q2} \\ 2^8\cdot 3\cdot 11
\end{array}
$
\\\hline
$
\begin{array}{c}
(-31,8) \\ 8712 \\ 2^3\cdot 3^2\cdot 11^2 \\ \textrm{2112r4} \\ 2^6\cdot 3\cdot 11
\end{array}
$
&
$
\begin{array}{c}
(-25,32) \\ 85536 \\ 2^5\cdot 3^5\cdot 11 \\ \textrm{66c2} \\ 2\cdot 3\cdot 11
\end{array}
$
&
$
\begin{array}{c}
(-22,1) \\ 486 \\ 2\cdot 3^5 \\ \textrm{768d4} \\ 2^8\cdot 3
\end{array}
$
&
$
\begin{array}{c}
(-19,1) \\ 363 \\ 3\cdot 11^2 \\ \textrm{4224s2} \\ 2^7\cdot 3\cdot 11
\end{array}
$
&
$
\begin{array}{c}
(-17,2) \\ 594 \\ 2\cdot 3^3\cdot 11 \\ \textrm{2112b2} \\ 2^6\cdot 3\cdot 11
\end{array}
$
&
$
\begin{array}{c}
(-14,1) \\ 198 \\ 2\cdot 3^2\cdot 11 \\ \textrm{8448f2} \\ 2^8\cdot 3\cdot 11
\end{array}
$
\\
\hline
\end{tabular}
\end{center}
\label{default}
\end{table}

\begin{table}
\caption{$(a,b,c,d)=(0,1,0,1)$, $\delta=-1 \cdot 2^2$, $ S = \{{2,3,7,11}\}$}\label{010123711}
\begin{center}
\begin{tabular}{|c|c|c|c|c|c|c|c|c|}
\hline
$
\begin{array}{c}
(-1,1) \\ 2 \\ 2 \\ \textrm{128c1} \\ 2^7
\end{array}
$
&
$
\begin{array}{c}
(0,1) \\ 1 \\ 1 \\ \textrm{256c2} \\ 2^8
\end{array}
$
&
$
\begin{array}{c}
(1,1) \\ 2 \\ 2 \\ \textrm{128c1} \\ 2^7
\end{array}
$
\\\hline
\end{tabular}
\end{center}
\label{default}
\end{table}

\begin{table}
\caption{$(a,b,c,d)=(0,1,0,1)$, $\delta=-1\cdot 2^2$, $ S = \{{2, 5, 13}\}$}\label{01012513}
\begin{center}
\footnotesize
\begin{tabular}{|c|c|c|c|c|c|c|c|c|}
\hline
$
\begin{array}{c}
(-239,1) \\ 57122 \\ 2\cdot 13^4 \\ \textrm{1664k2} \\ 2^7\cdot 13
\end{array}
$
&
$
\begin{array}{c}
(-199,32) \\ 1300000 \\ 2^5\cdot 5^5\cdot 13 \\ \textrm{4160i2} \\ 2^6\cdot 5\cdot 13
\end{array}
$
&
$
\begin{array}{c}
(-83,64) \\ 703040 \\ 2^6\cdot 5\cdot 13^3 \\ \textrm{130a4} \\ 2\cdot 5\cdot 13
\end{array}
$
&
$
\begin{array}{c}
(-63,16) \\ 67600 \\ 2^4\cdot 5^2\cdot 13^2 \\ \textrm{4160a4} \\ 2^6\cdot 5\cdot 13
\end{array}
$
&
$
\begin{array}{c}
(-57,1) \\ 3250 \\ 2\cdot 5^3\cdot 13 \\ \textrm{8320d2} \\ 2^7\cdot 5\cdot 13
\end{array}
$
&
$
\begin{array}{c}
(-37,16) \\ 26000 \\ 2^4\cdot 5^3\cdot 13 \\ \textrm{4160n2} \\ 2^6\cdot 5\cdot 13
\end{array}
$
\\\hline
$
\begin{array}{c}
(-29,2) \\ 1690 \\ 2\cdot 5\cdot 13^2 \\ \textrm{4160p2} \\ 2^6\cdot 5\cdot 13
\end{array}
$
&
$
\begin{array}{c}
(-18,1) \\ 325 \\ 5^2\cdot 13 \\ \textrm{2080b2} \\ 2^5\cdot 5\cdot 13
\end{array}
$
&
$
\begin{array}{c}
(-12,5) \\ 845 \\ 5\cdot 13^2 \\ \textrm{2080f4} \\ 2^5\cdot 5\cdot 13
\end{array}
$
&
$
\begin{array}{c}
(-11,2) \\ 250 \\ 2\cdot 5^3 \\ \textrm{320c4} \\ 2^6\cdot 5
\end{array}
$
&
$
\begin{array}{c}
(-9,13) \\ 3250 \\ 2\cdot 5^3\cdot 13 \\ \textrm{8320a1} \\ 2^7\cdot 5\cdot 13
\end{array}
$
&
$
\begin{array}{c}
(-8,1) \\ 65 \\ 5\cdot 13 \\ \textrm{2080c2} \\ 2^5\cdot 5\cdot 13
\end{array}
$
\\\hline
$
\begin{array}{c}
(-7,1) \\ 50 \\ 2\cdot 5^2 \\ \textrm{640f2} \\ 2^7\cdot 5
\end{array}
$
&
$
\begin{array}{c}
(-7,4) \\ 260 \\ 2^2\cdot 5\cdot 13 \\ \textrm{4160c2} \\ 2^6\cdot 5\cdot 13
\end{array}
$
&
$
\begin{array}{c}
(-5,1) \\ 26 \\ 2\cdot 13 \\ \textrm{1664d2} \\ 2^7\cdot 13
\end{array}
$
&
$
\begin{array}{c}
(-3,1) \\ 10 \\ 2\cdot 5 \\ \textrm{640a2} \\ 2^7\cdot 5
\end{array}
$
&
$
\begin{array}{c}
(-3,2) \\ 26 \\ 2\cdot 13 \\ \textrm{52a1} \\ 2^2\cdot 13
\end{array}
$
&
$
\begin{array}{c}
(-3,4) \\ 100 \\ 2^2\cdot 5^2 \\ \textrm{40a4} \\ 2^3\cdot 5
\end{array}
$
\\
\hline
\end{tabular}
\end{center}
\label{default}
\end{table}

\begin{table}
\caption{$(a,b,c,d)=(0,1,0,-2)$, $\delta=2^3$, $ S = \{{2, 5, 13,17}\}$} \label{010-2251317}
\begin{center}
\footnotesize
\begin{tabular}{|c|c|c|c|c|c|c|c|c|}
\hline
$
\begin{array}{c}
(-4855,3328) \\ 4725284096 \\ 2^8\cdot 13\cdot 17^5 \\ \textrm{14144v2} \\ 2^6\cdot 13\cdot 17
\end{array}
$
&
$
\begin{array}{c}
(-239,169) \\ -169 \\ -1\cdot 13^2 \\ \textrm{1664h2} \\ 2^7\cdot 13
\end{array}
$
&
$
\begin{array}{c}
(-71,8) \\ 39304 \\ 2^3\cdot 17^3 \\ \textrm{1088c4} \\ 2^6\cdot 17
\end{array}
$
&
$
\begin{array}{c}
(-37,26) \\ 442 \\ 2\cdot 13\cdot 17 \\ \textrm{14144s2} \\ 2^6\cdot 13\cdot 17
\end{array}
$
&
$
\begin{array}{c}
(-33,20) \\ 5780 \\ 2^2\cdot 5\cdot 17^2 \\ \textrm{5440h3} \\ 2^6\cdot 5\cdot 17
\end{array}
$
&
$
\begin{array}{c}
(-31,25) \\ -7225 \\ -1\cdot 5^2\cdot 17^2 \\ \textrm{10880f2} \\ 2^7\cdot 5\cdot 17
\end{array}
$
\\\hline
$
\begin{array}{c}
(-24,17) \\ -34 \\ -1\cdot 2\cdot 17 \\ \textrm{4352f2} \\ 2^8\cdot 17
\end{array}
$
&
$
\begin{array}{c}
(-23,16) \\ 272 \\ 2^4\cdot 17 \\ \textrm{1088i2} \\ 2^6\cdot 17
\end{array}
$
&
$
\begin{array}{c}
(-7,4) \\ 68 \\ 2^2\cdot 17 \\ \textrm{1088g2} \\ 2^6\cdot 17
\end{array}
$
&
$
\begin{array}{c}
(-7,5) \\ -5 \\ -1\cdot 5 \\ \textrm{640d2} \\ 2^7\cdot 5
\end{array}
$
&
$
\begin{array}{c}
(-7,13) \\ -3757 \\ -1\cdot 13\cdot 17^2 \\ \textrm{28288f2} \\ 2^7\cdot 13\cdot 17
\end{array}
$
&
$
\begin{array}{c}
(-6,1) \\ 34 \\ 2\cdot 17 \\ \textrm{4352d2} \\ 2^8\cdot 17
\end{array}
$
\\\hline
$
\begin{array}{c}
(-5,2) \\ 34 \\ 2\cdot 17 \\ \textrm{1088j2} \\ 2^6\cdot 17
\end{array}
$
&
$
\begin{array}{c}
(-4,5) \\ -170 \\ -1\cdot 2\cdot 5\cdot 17 \\ \textrm{21760q2} \\ 2^8\cdot 5\cdot 17
\end{array}
$
&
$
\begin{array}{c}
(-3,2) \\ 2 \\ 2 \\ \textrm{64a2} \\ 2^6
\end{array}
$
&
$
\begin{array}{c}
(-2,1) \\ 2 \\ 2 \\ \textrm{256a2} \\ 2^8
\end{array}
$
&
$
\begin{array}{c}
(-1,1) \\ -1 \\ -1 \\ \textrm{128c2} \\ 2^7
\end{array}
$
&
$
\begin{array}{c}
(0,1) \\ -2 \\ -1\cdot 2 \\ \textrm{256c2} \\ 2^8
\end{array}
$
\\
\hline
\end{tabular}
\end{center}
\label{default}
\end{table}

\begin{table}
\caption{$(a,b,c,d)=(0,1,0,-2)$, $\delta=2^3$, $ S = \{{2, 7,29}\}$}\label{010-22729}
\begin{center}
\footnotesize
\begin{tabular}{|c|c|c|c|c|c|c|c|c|}
\hline
$
\begin{array}{c}
(-181,128) \\ -896 \\ -1\cdot 2^7\cdot 7 \\ \textrm{448c6} \\ 2^6\cdot 7
\end{array}
$
&
$
\begin{array}{c}
(-163,116) \\ -39788 \\ -1\cdot 2^2\cdot 7^3\cdot 29 \\ \textrm{12992bc2} \\ 2^6\cdot 7\cdot 29
\end{array}
$
&
$
\begin{array}{c}
(-45,29) \\ 9947 \\ 7^3\cdot 29 \\ \textrm{25984i2} \\ 2^7\cdot 7\cdot 29
\end{array}
$
&
$
\begin{array}{c}
(-41,29) \\ -29 \\ -1\cdot 29 \\ \textrm{3712p2} \\ 2^7\cdot 29
\end{array}
$
&
$
\begin{array}{c}
(-13,16) \\ -5488 \\ -1\cdot 2^4\cdot 7^3 \\ \textrm{448c4} \\ 2^6\cdot 7
\end{array}
$
&
$
\begin{array}{c}
(-11,8) \\ -56 \\ -1\cdot 2^3\cdot 7 \\ \textrm{448f2} \\ 2^6\cdot 7
\end{array}
$
\\\hline
$
\begin{array}{c}
(-10,1) \\ 98 \\ 2\cdot 7^2 \\ \textrm{1792f2} \\ 2^8\cdot 7
\end{array}
$
&
$
\begin{array}{c}
(-10,7) \\ 14 \\ 2\cdot 7 \\ \textrm{1792a2} \\ 2^8\cdot 7
\end{array}
$
&
$
\begin{array}{c}
(-9,4) \\ 196 \\ 2^2\cdot 7^2 \\ \textrm{448a3} \\ 2^6\cdot 7
\end{array}
$
&
$
\begin{array}{c}
(-5,4) \\ -28 \\ -1\cdot 2^2\cdot 7 \\ \textrm{448h2} \\ 2^6\cdot 7
\end{array}
$
&
$
\begin{array}{c}
(-4,1) \\ 14 \\ 2\cdot 7 \\ \textrm{1792e2} \\ 2^8\cdot 7
\end{array}
$
&
$
\begin{array}{c}
(-3,1) \\ 7 \\ 7 \\ \textrm{896d2} \\ 2^7\cdot 7
\end{array}
$
\\\hline
$
\begin{array}{c}
(-3,2) \\ 2 \\ 2 \\ \textrm{64a2} \\ 2^6
\end{array}
$
&
$
\begin{array}{c}
(-2,1) \\ 2 \\ 2 \\ \textrm{256a2} \\ 2^8
\end{array}
$
&
$
\begin{array}{c}
(-1,1) \\ -1 \\ -1 \\ \textrm{128c2} \\ 2^7
\end{array}
$
&
$
\begin{array}{c}
(-1,2) \\ -14 \\ -1\cdot 2\cdot 7 \\ \textrm{448d2} \\ 2^6\cdot 7
\end{array}
$
&
$
\begin{array}{c}
(0,1) \\ -2 \\ -1\cdot 2 \\ \textrm{256c2} \\ 2^8
\end{array}
$
&
$
\begin{array}{c}
(1,1) \\ -1 \\ -1 \\ \textrm{128c2} \\ 2^7
\end{array}
$
\\
\hline
\end{tabular}
\end{center}
\label{default}
\end{table}

\begin{table}
\caption{$(a,b,c,d)=(0,1,0,-3)$, $\delta=2^2\cdot 3$, $ S = \{{2,3,7,11}\}$}\label{010-323711}
\begin{center}
\tiny
\begin{tabular}{|c|c|c|c|c|c|c|c|c|}
\hline
$
\begin{array}{c}
(-122,9) \\ 131769 \\ 3^2\cdot 11^4 \\ \textrm{3168z3} \\ 2^5\cdot 3^2\cdot 11
\end{array}
$
&
$
\begin{array}{c}
(-111,64) \\ 2112 \\ 2^6\cdot 3\cdot 11 \\ \textrm{6336b4} \\ 2^6\cdot 3^2\cdot 11
\end{array}
$
&
$
\begin{array}{c}
(-97,56) \\ 56 \\ 2^3\cdot 7 \\ \textrm{4032bm5} \\ 2^6\cdot 3^2\cdot 7
\end{array}
$
&
$
\begin{array}{c}
(-85,49) \\ 1078 \\ 2\cdot 7^2\cdot 11 \\ \textrm{88704z2} \\ 2^7\cdot 3^2\cdot 7\cdot 11
\end{array}
$
&
$
\begin{array}{c}
(-69,16) \\ 63888 \\ 2^4\cdot 3\cdot 11^3 \\ \textrm{6336bi4} \\ 2^6\cdot 3^2\cdot 11
\end{array}
$
&
$
\begin{array}{c}
(-53,7) \\ 18634 \\ 2\cdot 7\cdot 11^3 \\ \textrm{88704cs2} \\ 2^7\cdot 3^2\cdot 7\cdot 11
\end{array}
$
\\\hline
$
\begin{array}{c}
(-47,27) \\ 594 \\ 2\cdot 3^3\cdot 11 \\ \textrm{12672i2} \\ 2^7\cdot 3^2\cdot 11
\end{array}
$
&
$
\begin{array}{c}
(-43,24) \\ 2904 \\ 2^3\cdot 3\cdot 11^2 \\ \textrm{6336bd3} \\ 2^6\cdot 3^2\cdot 11
\end{array}
$
&
$
\begin{array}{c}
(-38,21) \\ 2541 \\ 3\cdot 7\cdot 11^2 \\ \textrm{22176f4} \\ 2^5\cdot 3^2\cdot 7\cdot 11
\end{array}
$
&
$
\begin{array}{c}
(-31,18) \\ -198 \\ -1\cdot 2\cdot 3^2\cdot 11 \\ \textrm{6336m2} \\ 2^6\cdot 3^2\cdot 11
\end{array}
$
&
$
\begin{array}{c}
(-27,1) \\ 726 \\ 2\cdot 3\cdot 11^2 \\ \textrm{12672c2} \\ 2^7\cdot 3^2\cdot 11
\end{array}
$
&
$
\begin{array}{c}
(-19,11) \\ -22 \\ -1\cdot 2\cdot 11 \\ \textrm{12672bc2} \\ 2^7\cdot 3^2\cdot 11
\end{array}
$
\\
\hline
\end{tabular}
\end{center}
\label{default}
\end{table}

\begin{table}
\caption{$(a,b,c,d)=(0,1,0,-7)$, $\delta=2^2\cdot 7$, $ S = \{{2, 5, 7, 11}\}$}\label{010-725711}
\begin{center}
\begin{tabular}{|c|c|c|c|c|c|c|c|c|}
\hline
$
\begin{array}{c}
(-21,8) \\ -56 \\ -1\cdot 2^3\cdot 7 \\ \textrm{3136r4} \\ 2^6\cdot 7^2
\end{array}
$
&
$
\begin{array}{c}
(-3,1) \\ 2 \\ 2 \\ \textrm{6272b2} \\ 2^7\cdot 7^2
\end{array}
$
&
$
\begin{array}{c}
(0,1) \\ -7 \\ -1\cdot 7 \\ \textrm{12544a2} \\ 2^8\cdot 7^2
\end{array}
$
&
$
\begin{array}{c}
(3,1) \\ 2 \\ 2 \\ \textrm{6272b2} \\ 2^7\cdot 7^2
\end{array}
$
&
$
\begin{array}{c}
(21,8) \\ -56 \\ -1\cdot 2^3\cdot 7 \\ \textrm{3136r4} \\ 2^6\cdot 7^2
\end{array}
$
\\\hline
\end{tabular}
\end{center}
\label{default}
\end{table}

\begin{table}
\caption{$(a,b,c,d)=(1,-1,-4,-1)$, $\delta=13^2$, $ S = \{{2, 5,13}\}$}\label{1-1-4-12513}
\begin{center}
\footnotesize
\begin{tabular}{|c|c|c|c|c|c|c|c|c|}
\hline
$
\begin{array}{c}
(-157,114) \\ 65 \\ 5\cdot13 \\ \textrm{54080bk2} \\ 2^6\cdot5\cdot13^2
\end{array}
$
&
$
\begin{array}{c}
(-43,157) \\ -65 \\ -1\cdot5\cdot13 \\ \textrm{54080bk2} \\ 2^6\cdot5\cdot13^2
\end{array}
$
&
$
\begin{array}{c}
(-17,13) \\ 625 \\ 5^4 \\ \textrm{54080di1} \\ 2^6\cdot5\cdot13^2
\end{array}
$
&
$
\begin{array}{c}
(-17,14) \\ 1625 \\ 5^3\cdot13 \\ \textrm{54080bk1} \\ 2^6\cdot5\cdot13^2
\end{array}
$
&
$
\begin{array}{c}
(-11,8) \\ 5 \\ 5 \\ \textrm{54080cz1} \\ 2^6\cdot5\cdot13^2
\end{array}
$
&
$
\begin{array}{c}
(-9,7) \\ 125 \\ 5^3 \\ \textrm{54080bl2} \\ 2^6\cdot5\cdot13^2
\end{array}
$
\\\hline
$
\begin{array}{c}
(-7,5) \\ -13 \\ -1\cdot13 \\ \textrm{10816bi1} \\ 2^6\cdot13^2
\end{array}
$
&
$
\begin{array}{c}
(-4,1) \\ -65 \\ -1\cdot5\cdot13 \\ \textrm{54080cd1} \\ 2^6\cdot5\cdot13^2
\end{array}
$
&
$
\begin{array}{c}
(-4,3) \\ 5 \\ 5 \\ \textrm{54080cb1} \\ 2^6\cdot5\cdot13^2
\end{array}
$
&
$
\begin{array}{c}
(-4,17) \\ -625 \\ -1\cdot5^4 \\ \textrm{54080di1} \\ 2^6\cdot5\cdot13^2
\end{array}
$
&
$
\begin{array}{c}
(-3,1) \\ -25 \\ -1\cdot5^2 \\ \textrm{54080dc1} \\ 2^6\cdot5\cdot13^2
\end{array}
$
&
$
\begin{array}{c}
(-3,2) \\ -5 \\ -1\cdot5 \\ \textrm{54080cx1} \\ 2^6\cdot5\cdot13^2
\end{array}
$
\\\hline
$
\begin{array}{c}
(-3,4) \\ 65 \\ 5\cdot13 \\ \textrm{54080cd1} \\ 2^6\cdot5\cdot13^2
\end{array}
$
&
$
\begin{array}{c}
(-3,11) \\ -5 \\ -1\cdot5 \\ \textrm{54080cz1} \\ 2^6\cdot5\cdot13^2
\end{array}
$
&
$
\begin{array}{c}
(-3,17) \\ -1625 \\ -1\cdot5^3\cdot13 \\ \textrm{54080bk1} \\ 2^6\cdot5\cdot13^2
\end{array}
$
&
$
\begin{array}{c}
(-2,1) \\ -5 \\ -1\cdot5 \\ \textrm{54080co1} \\ 2^6\cdot5\cdot13^2
\end{array}
$
&
$
\begin{array}{c}
(-2,3) \\ 25 \\ 5^2 \\ \textrm{54080dc1} \\ 2^6\cdot5\cdot13^2
\end{array}
$
&
$
\begin{array}{c}
(-2,7) \\ 13 \\ 13 \\ \textrm{10816bi1} \\ 2^6\cdot13^2
\end{array}
$
\\
\hline

\end{tabular}
\end{center}
\label{default}

\end{table}

\begin{table}
\caption{$(a,b,c,d)=(1,-1,-2,-2)$, $\delta=-1\cdot 2^3\cdot 19$, $ S = \{{2, 5,19}\}$}\label{1-1-2-22519}\label{1-1-2-22519}
\begin{center}
\tiny
\begin{tabular}{|c|c|c|c|c|c|c|c|c|}
\hline
$
\begin{array}{c}
(-81,13) \\ -593750 \\ -1\cdot 2\cdot 5^6\cdot 19 \\ \textrm{115520ck1} \\ 2^6\cdot 5\cdot 19^2
\end{array}
$
&
$
\begin{array}{c}
(-43,89) \\ -972800 \\ -1\cdot 2^{11}\cdot 5^2\cdot 19 \\ \textrm{3610i1} \\ 2\cdot 5\cdot 19^2
\end{array}
$
&
$
\begin{array}{c}
(-17,16) \\ -9025 \\ -1\cdot 5^2\cdot 19^2 \\ \textrm{12160i1} \\ 2^7\cdot 5\cdot 19
\end{array}
$
&
$
\begin{array}{c}
(-8,9) \\ -1250 \\ -1\cdot 2\cdot 5^4 \\ \textrm{24320s1} \\ 2^8\cdot 5\cdot 19
\end{array}
$
&
$
\begin{array}{c}
(-7,1) \\ -380 \\ -1\cdot 2^2\cdot 5\cdot 19 \\ \textrm{115520q1} \\ 2^6\cdot 5\cdot 19^2
\end{array}
$
&
$
\begin{array}{c}
(-3,1) \\ -32 \\ -1\cdot 2^5 \\ \textrm{23104r1} \\ 2^6\cdot 19^2
\end{array}
$
\\\hline
$
\begin{array}{c}
(-3,4) \\ -95 \\ -1\cdot 5\cdot 19 \\ \textrm{231040u1} \\ 2^7\cdot 5\cdot 19^2
\end{array}
$
&
$
\begin{array}{c}
(-2,1) \\ -10 \\ -1\cdot 2\cdot 5 \\ \textrm{24320j1} \\ 2^8\cdot 5\cdot 19
\end{array}
$
&
$
\begin{array}{c}
(-2,3) \\ -38 \\ -1\cdot 2\cdot 19 \\ \textrm{4864p1} \\ 2^8\cdot 19
\end{array}
$
&
$
\begin{array}{c}
(-1,1) \\ -2 \\ -1\cdot 2 \\ \textrm{608f1} \\ 2^5\cdot 19
\end{array}
$
&
$
\begin{array}{c}
(-1,3) \\ -40 \\ -1\cdot 2^3\cdot 5 \\ \textrm{190b1} \\ 2\cdot 5\cdot 19
\end{array}
$
&
$
\begin{array}{c}
(-1,11) \\ -2432 \\ -1\cdot 2^7\cdot 19 \\ \textrm{38a1} \\ 2\cdot 19
\end{array}
$
\\
\hline
\end{tabular}
\end{center}
\label{default}
\end{table}

\begin{table}
\caption{$(a,b,c,d)=(1,0,0,1)$, $\delta=-1\cdot 3^3$, $ S = \{{2, 3, 5}\}$} \label{1001235}
\begin{center}
\begin{tabular}{|c|c|c|c|c|c|c|c|c|}
\hline
$
\begin{array}{c}
(0,1) \\ 1 \\ 1 \\ \textrm{15552b2} \\ 2^6\cdot 3^5
\end{array}
$
&
$
\begin{array}{c}
(1,1) \\ 2 \\ 2 \\ \textrm{2304j2} \\ 2^8\cdot 3^2
\end{array}
$
&
$
\begin{array}{c}
(1,2) \\ 9 \\ 3^2 \\ \textrm{48a4} \\ 2^4\cdot 3
\end{array}
$
&
$
\begin{array}{c}
(2,1) \\ 9 \\ 3^2 \\ \textrm{48a4} \\ 2^4\cdot 3
\end{array}
$
&
empty
& empty
\\\hline
\end{tabular}
\end{center}
\label{default}
\caption{$(a,b,c,d)=(1,0,0,2)$, $\delta=-1\cdot 2^2\cdot 3^3$, $ S = \{{2, 3, 5}\}$}\label{1002235}
\begin{center}
\begin{tabular}{|c|c|c|c|c|c|c|c|c|}
\hline
$
\begin{array}{c}
(-37,29) \\ -1875 \\ -1\cdot 3\cdot 5^4 \\ \textrm{4320c1} \\ 2^5\cdot 3^3\cdot 5
\end{array}
$
&
$
\begin{array}{c}
(-5,4) \\ 3 \\ 3 \\ \textrm{432a4} \\ 2^4\cdot 3^3
\end{array}
$
&
$
\begin{array}{c}
(-4,3) \\ -10 \\ -1\cdot 2\cdot 5 \\ \textrm{8640cb1} \\ 2^6\cdot 3^3\cdot 5
\end{array}
$
&
$
\begin{array}{c}
(-3,1) \\ -25 \\ -1\cdot 5^2 \\ \textrm{4320j1} \\ 2^5\cdot 3^3\cdot 5
\end{array}
$
&
$
\begin{array}{c}
(-2,1) \\ -6 \\ -1\cdot 2\cdot 3 \\ \textrm{1728l1} \\ 2^6\cdot 3^3
\end{array}
$
&
$
\begin{array}{c}
(-1,1) \\ 1 \\ 1 \\ \textrm{864i1} \\ 2^5\cdot 3^3
\end{array}
$
\\\hline
$
\begin{array}{c}
(-1,2) \\ 15 \\ 3\cdot 5 \\ \textrm{2160w1} \\ 2^4\cdot 3^3\cdot 5
\end{array}
$
&
$
\begin{array}{c}
(0,1) \\ 2 \\ 2 \\ \textrm{15552b2} \\ 2^6\cdot 3^5
\end{array}
$
&
$
\begin{array}{c}
(1,1) \\ 3 \\ 3 \\ \textrm{1728h1} \\ 2^6\cdot 3^3
\end{array}
$
&
$
\begin{array}{c}
(2,1) \\ 10 \\ 2\cdot 5 \\ \textrm{1080j1} \\ 2^3\cdot 3^3\cdot 5
\end{array}
$
&
$
\begin{array}{c}
(4,7) \\ 750 \\ 2\cdot 3\cdot 5^3 \\ \textrm{540e2} \\ 2^2\cdot 3^3\cdot 5
\end{array}
$
&
empty
\\\hline
\end{tabular}
\end{center}
\label{default}
\caption{$(a,b,c,d)=(1,0,0,-2)$, $\delta=-1\cdot 2^2\cdot 3^3$, $ S = \{{2, 3, 5}\}$}\label{100-2235}
\begin{center}
\begin{tabular}{|c|c|c|c|c|c|c|c|c|}
\hline
$
\begin{array}{c}
(-4,7) \\ -750 \\ -1\cdot 2\cdot 3\cdot 5^3 \\ \textrm{540e2} \\ 2^2\cdot 3^3\cdot 5
\end{array}
$
&
$
\begin{array}{c}
(-2,1) \\ -10 \\ -1\cdot 2\cdot 5 \\ \textrm{1080j1} \\ 2^3\cdot 3^3\cdot 5
\end{array}
$
&
$
\begin{array}{c}
(-1,1) \\ -3 \\ -1\cdot 3 \\ \textrm{1728h1} \\ 2^6\cdot 3^3
\end{array}
$
&
$
\begin{array}{c}
(0,1) \\ -2 \\ -1\cdot 2 \\ \textrm{15552b2} \\ 2^6\cdot 3^5
\end{array}
$
&
$
\begin{array}{c}
(1,1) \\ -1 \\ -1 \\ \textrm{864i1} \\ 2^5\cdot 3^3
\end{array}
$
&
$
\begin{array}{c}
(1,2) \\ -15 \\ -1\cdot 3\cdot 5 \\ \textrm{2160w1} \\ 2^4\cdot 3^3\cdot 5
\end{array}
$
\\\hline
$
\begin{array}{c}
(2,1) \\ 6 \\ 2\cdot 3 \\ \textrm{1728l1} \\ 2^6\cdot 3^3
\end{array}
$
&
$
\begin{array}{c}
(3,1) \\ 25 \\ 5^2 \\ \textrm{4320j1} \\ 2^5\cdot 3^3\cdot 5
\end{array}
$
&
$
\begin{array}{c}
(4,3) \\ 10 \\ 2\cdot 5 \\ \textrm{8640cb1} \\ 2^6\cdot 3^3\cdot 5
\end{array}
$
&
$
\begin{array}{c}
(5,4) \\ -3 \\ -1\cdot 3 \\ \textrm{432a4} \\ 2^4\cdot 3^3
\end{array}
$
&
$
\begin{array}{c}
(37,29) \\ 1875 \\ 3\cdot 5^4 \\ \textrm{4320c1} \\ 2^5\cdot 3^3\cdot 5
\end{array}
$
&
empty
\\\hline
\end{tabular}
\end{center}
\label{default}
\end{table}
\begin{table}
\caption{$(a,b,c,d)=(1,0,0,-3)$, $\delta=-1\cdot 3^5$, $ S = \{{2, 3, 5}\}$}\label{100-3235}
\begin{center}
\begin{tabular}{|c|c|c|c|c|c|c|c|c|}
\hline
$
\begin{array}{c}
(-21,17) \\ -24000 \\ -1\cdot 2^6\cdot 3\cdot 5^3 \\ \textrm{2430g1} \\ 2\cdot 3^5\cdot 5
\end{array}
$
&
$
\begin{array}{c}
(-5,1) \\ -128 \\ -1\cdot 2^7 \\ \textrm{486a1} \\ 2\cdot 3^5
\end{array}
$
&
$
\begin{array}{c}
(-3,1) \\ -30 \\ -1\cdot 2\cdot 3\cdot 5 \\ \textrm{38880o1} \\ 2^5\cdot 3^5\cdot 5
\end{array}
$
&
$
\begin{array}{c}
(-1,1) \\ -4 \\ -1\cdot 2^2 \\ \textrm{1944j1} \\ 2^3\cdot 3^5
\end{array}
$
&
$
\begin{array}{c}
(-1,2) \\ -25 \\ -1\cdot 5^2 \\ \textrm{19440m1} \\ 2^4\cdot 3^5\cdot 5
\end{array}
$
&
$
\begin{array}{c}
(0,1) \\ -3 \\ -1\cdot 3 \\ \textrm{15552b2} \\ 2^6\cdot 3^5
\end{array}
$
\\\hline
$
\begin{array}{c}
(1,1) \\ -2 \\ -1\cdot 2 \\ \textrm{15552k1} \\ 2^6\cdot 3^5
\end{array}
$
&
$
\begin{array}{c}
(1,3) \\ -80 \\ -1\cdot 2^4\cdot 5 \\ \textrm{77760q1} \\ 2^6\cdot 3^5\cdot 5
\end{array}
$
&
$
\begin{array}{c}
(2,1) \\ 5 \\ 5 \\ \textrm{77760y1} \\ 2^6\cdot 3^5\cdot 5
\end{array}
$
&
$
\begin{array}{c}
(3,1) \\ 24 \\ 2^3\cdot 3 \\ \textrm{15552bw1} \\ 2^6\cdot 3^5
\end{array}
$
&
$
\begin{array}{c}
(3,2) \\ 3 \\ 3 \\ \textrm{15552bd1} \\ 2^6\cdot 3^5
\end{array}
$
&
$
\begin{array}{c}
(7,5) \\ -32 \\ -1\cdot 2^5 \\ \textrm{15552bo2} \\ 2^6\cdot 3^5
\end{array}
$
\\\hline
$
\begin{array}{c}
(9,7) \\ -300 \\ -1\cdot 2^2\cdot 3\cdot 5^2 \\ \textrm{77760v1} \\ 2^6\cdot 3^5\cdot 5
\end{array}
$
&
$
\begin{array}{c}
(11,3) \\ 1250 \\ 2\cdot 5^4 \\ \textrm{77760bv1} \\ 2^6\cdot 3^5\cdot 5
\end{array}
$
&
$
\begin{array}{c}
(13,9) \\ 10 \\ 2\cdot 5 \\ \textrm{77760cs1} \\ 2^6\cdot 3^5\cdot 5
\end{array}
$
&
$
\begin{array}{c}
(33,19) \\ 15360 \\ 2^{10}\cdot 3\cdot 5 \\ \textrm{77760b2} \\ 2^6\cdot 3^5\cdot 5
\end{array}
$
&
empty
& empty
\\\hline
\end{tabular}
\end{center}
\label{default}
\end{table}

\end{document}